\documentclass[letterpaper,11pt]{amsart}

\usepackage{graphicx,color,epsfig}
\usepackage{amsfonts,amsmath,amssymb,amsthm}
\usepackage[normalem]{ulem}
\usepackage{comment,constants}
\usepackage{hyperref}
\usepackage{mathtools}

\newtheorem{thm}{Theorem}[section]
\newtheorem{prop}[thm]{Proposition}
\newtheorem{lem}[thm]{Lemma}
\newtheorem{cor}[thm]{Corollary}

\newtheorem*{asm*}{Assumptions}
\newtheorem{op}[thm]{Question}

\theoremstyle{remark}
\newtheorem{rem}[thm]{Remark}
\newtheorem*{rem*}{Remark}

\theoremstyle{definition}

\newtheorem{ex}[thm]{Example}

\newcommand{\ra}{\rightarrow}
\newcommand{\Ra}{\Rightarrow}

\newcommand{\N}{\mathbb N}     
\newcommand{\R}{\mathbb R}     
\newcommand{\Z}{\mathbb Z}     

\newcommand{\cal}{\mathcal}

\renewcommand{\a}{\alpha}
\renewcommand{\b}{\beta}

\renewcommand{\d}{\delta}

\newcommand{\e}{\varepsilon}

\renewcommand{\l}{\lambda}

\newcommand{\s}{\sigma}

\renewcommand{\epsilon}{\varepsilon}

\newcommand{\bigo}{\mathcal{O}}

\newcommand{\fl}[1]{\lfloor #1 \rfloor}  
\newcommand{\ind}[1]{ \mathbf{1}_{ \{ #1 \} } } 

\DeclarePairedDelimiter{\ceil}{\lceil}{\rceil}
\DeclarePairedDelimiter{\floor}{\lfloor}{\rfloor}

\newconstantfamily{c}{symbol=c}

\newcommand{\w}{\omega}              
\renewcommand{\P}{\mathbb{P}}        
\newcommand{\E}{\mathbb{E}}          
\newcommand{\vp}{v_0}                
 
\DeclareMathOperator{\Var}{Var}
\DeclareMathOperator{\Cov}{Cov}
\DeclareMathOperator{\sgn}{sign}

\newcommand{\bp}{\mathbf{p}}

\newcommand{\br}{\mathbf{r}}
\newcommand{\one}{\mathbf{1}}
\newcommand{\ve}[1]{\mathbf{#1}}

\title[Excited random walks]{Excited random walks with Markovian cookie stacks}

\author{Elena Kosygina}
\address{Elena Kosygina\\One Bernard Baruch Way \\ Department of Mathematics, Box B6-230 \\ Baruch College \\ New York, NY 10010 \\ USA}
\email{elena.kosygina@baruch.cuny.edu}
\urladdr{www.baruch.cuny.edu/math/facultystaff/ekosygina/kosygina.html}

\thanks{E.\,Kosygina was partially supported by the Simons Foundation through a Collaboration Grant for Mathematicians \#209493 and Simons Fellowship in Mathematics, 2014-2015. She also thanks Institut Mittag-Leffler for support and stimulating research environment.}

\author{Jonathon Peterson}
\address{Jonathon Peterson\\Purdue University\\Department of Mathematics\\150 N University Street\\West Lafayette, IN  47907\\USA}
\email{peterson@purdue.edu}
\urladdr{www.math.purdue.edu/~peterson}
\thanks{J. Peterson was partially supported by NSA grant H98230-13-1-0266.}


\begin{document}

\begin{abstract}
We consider a nearest-neighbor random walk on $\mathbb{Z}$
whose probability $\omega_x(j)$ to jump to the right from site $x$
depends not only on $x$ but also on the number of prior visits $j$ to
$x$. The collection $(\omega_x(j))_{x\in\mathbb{Z},j\ge 1}$ is
sometimes called the ``cookie environment'' due to the following
informal interpretation. Upon each visit to a site the walker eats a cookie
from the cookie stack at that site and chooses the transition
probabilities according to the ``strength'' of the cookie eaten.
We assume that the cookie stacks are i.i.d.\ and that the cookie
``strengths'' within the stack $(\omega_x(j))_{j\ge 1}$ at site $x$
follow a finite state Markov chain. Thus, the
environment at each site is dynamic, but it evolves according to the
local time of the walk at each site rather than the original random
walk time.  

The model admits two different regimes, critical or non-critical,
depending on whether the expected probability to jump to the right (or
left) under the invariant measure for the Markov chain is equal to
$1/2$ or not. We show that in the non-critical regime the walk is
always transient, has non-zero linear speed, and satisfies the
classical central limit theorem. The critical regime allows for a much
more diverse behavior. We give necessary and sufficient conditions for
recurrence/transience and ballisticity of the walk in the critical
regime as well as a complete characterization of limit laws under the
averaged measure in the transient case.

The setting considered in this paper generalizes the previously
studied model with periodic cookie stacks \cite{kosERWPC}. Our results
on ballisticity and limit theorems are new even for the periodic
model.
\end{abstract}

\maketitle

\section{Introduction}

Excited random walks (ERWs, also called cookie random walks) are a
model for a self-interacting random
motion where the self-interaction is such that the transition
probabilities for the next step of the random walk depend on the local
time at the current location.  To make this precise, a \emph{cookie environment}
$\w = \{ \w_x(j) \}_{x\in \Z, \, j\geq 1}$ is an element of
$\Omega = [0,1]^{\Z\times \N}$, and given a cookie environment
$\w \in \Omega$ and $x \in \Z$, the ERW in the cookie
environment $\w$ started at $x$ is a stochastic process
$\{X_n\}_{n\geq 0}$ with law $P_\w^x$ such that $P_\w^x(X_0 = x)=1$ and
\begin{align*}
 P_\w^x\left( X_{n+1} = X_n+1 \, | \, X_0,X_1,\dots,X_n \right)
&= 1 - P_\w^x\left( X_{n+1} = X_n - 1 \, | \, X_0,X_1,\dots,X_n \right) \\
&= \w_{X_n}( \#\{k\leq n:\, X_k = X_n\} ).
\end{align*}
The ``cookie'' terminology in the above description of ERWs comes from the following interpretation. One imagines a stack of cookies at each site $x\in \Z$. Upon the $j$-th visit to the site $x$ the random walker eats the $j$-th cookie at that site which induces a step to the right with probability $\w_x(j)$ and to the left with probability $1-\w_x(j)$. For this reason we will refer to $\w_x(j)$ as the \emph{strength} of the $j$-th cookie at site $x$. 
 
Many results for ERWs in one dimension were obtained
under the assumption that there is an $M<\infty$ such that
$\w_x(j) = 1/2$ for all $j>M$ and $x\in\Z$
\cite{bwERW,mpvCRWspeed,bsCRWspeed,bsRGCRW,kzPNERW,kmLLCRW,dkSLRERW,pERWLDP,pCRWMnt,aoEM,pXSDERW}.
We will refer to this as the case of ``boundedly many cookies per
site''.  
Much less is known if the number of cookies per site is not bounded (in particular, if there are infinitely many cookies at each site).   
Notable exceptions are that Zerner
\cite{zMERW} proved a criterion for recurrence/transience and
Dolgopyat \cite{dCLTERW} proved the scaling limits for the recurrent
case under the assumption that $\w_x(j) \geq 1/2$ for all $x\in \Z$
and $j\geq 1$ (i.e., all cookies have non-negative drift).  For a
review (prior to 2012) of ERWs in one and more dimensions see
\cite{kzERWsurvey}.

Until recently, very little was known for ERWs which had infinitely
many cookies per site with both positive and negative drift
cookies.\footnote{With the exception of one-dimensional random walk in
  random environment which can be seen as a particular case of ERW
  where $\w_x(j) = \w_x(1)$ for all $j\geq 1$ at each $x\in \Z$.}
However, \cite{kosERWPC} gave an explicit criterion for
recurrence/transience of ERWs in cookie environments such that the
cookie sequence $\{\w_x(j)\}_{j\geq 1}$ is the same for all $x$ and is
periodic in $j$.  The results in the current paper cover more general cookie environments where the cookie
stack at each site is given by a finite state Markov chain. 
Moreover, we add to
the results of \cite{kosERWPC} by also proving a criterion for
ballisticity (non-zero limiting velocity) and limiting distributions
in the transient case.  We also note that our model is general enough
to include the case of periodic cookie 
stacks as in \cite{kosERWPC} as well as some cases of
boundedly many cookies per site.

\subsection{Description of the model}

Let $\mathcal{R} = \{1,2,\ldots,N\}$ denote a finite state space, and
let $\{R_j\}_{j\geq 1}$ be a Markov chain on $\mathcal{R}$ with
transition probabilities given by an $N\times N$ matrix
$K=(K_{i,i'})_{i,i' \in \mathcal{R}}$.  We will assume that this
Markov chain has a unique closed irreducible subset
$\mathcal{R}_0 \subseteq \mathcal{R}$.  Thus, there is a unique
stationary distribution $\mu$ for the Markov chain
$\{R_j\}_{j\geq 1}$, and $\mu$ is supported on
  ${\cal R}_0$.  Throughout the paper we will often
represent probability distributions
$\eta$ on $\mathcal{R}$ as row vectors
$\eta = (\eta(1),\eta(2),\ldots,\eta(N))$, so that for the stationary distribution $\mu$ we have
$\mu K = \mu$.

\begin{asm*}[i.i.d., elliptic, Markovian cookie stacks]
  Let $\{\mathbf{R}^x\}_{x\in\Z}$ be an i.i.d.\ family of Mar\-kov
  chains $\mathbf{R}^x = \{R_j^x\}_{j\geq 1}$ with transition matrix
  $K$. We assume that the cookie environment $\w$ is given by
  $\w_x(j) = p(R_j^x)$ for some fixed function
  $p:\mathcal{R} \ra (0,1)$. The requirement that $p$ is strictly
  between $0$ and $1$ will be referred to as ellipticity throughout
  the paper.
\end{asm*}
For any distribution $\eta$ on $\mathcal{R}$, let $\P_\eta$ denote the
distribution of $\{R^x_j\}_{x\in \Z,j\ge 1}$ when each $R_1^x$,
$x\in\Z$, has distribution $\eta$. With a slight abuse of notation
$\P_\eta$ will also be used for the induced distribution on
environments $\w$ which is constructed as in the assumptions above.
If $\eta$ is concentrated on a single state $i\in{\cal R}$, i.e.\ if $\eta=\delta_i$, we shall use $\P_i$ instead of $\P_{\delta_i}$.
Expectations with respect to $\P_\eta$ or $\P_i$
will be denoted by $\E_\eta$ and $\E_i$, respectively.

As stated above, the law of the ERW in a
fixed cookie environment $\omega$ is denoted
$P_\w^x$. We will refer to this as the \emph{quenched} law. If the
cookie environment has distribution $\P_\eta$, then we can also define
the \emph{averaged} law of the ERW by
$P_\eta^x(\cdot) = \E_\eta[ P_\w^x(\cdot) ]$.  Expectations with
respect to the quenched and averaged measures will be denoted by
$E_\w^x$ and $E_\eta^x$, respectively.  Again, if $\eta$ is concentrated on
$i \in \mathcal{R}$ we shall set
$P_i^x(\cdot): = \E_i[ P_\w^x(\cdot)]$.  We will often be interested
only in ERWs started at $X_0 = 0$, and so we will use the notation
$P_\w$, $P_\eta$, or $P_i$ in place of $P_\w^0$, $P_\eta^0$ or
$P_i^0$, respectively.

\begin{ex}[Periodic cookie sequences]\label{ex:period1}
  The model above clearly generalizes the case of periodic cookie
  sequences at each site as in \cite{kosERWPC}.  To obtain periodic
  cookie sequences set $K_{i,i+1} = 1$ for $i=1,\ldots,N-1$,
  $K_{N,1} = 1$, and let $\eta = \d_1$.
\end{ex}

\begin{ex}[Cookies stacks of geometric height]\label{geom}
  Fix $\alpha\in(0,1]$, $p(1)>1/2$, and $p(2)= 1/2$. Let
  $K = \begin{pmatrix} 1-\a & \a \\ 0 & 1 \end{pmatrix}$. If
  $\eta = \d_1$ then the cookie environment is such that there are an
  i.i.d.\ Geometric($\a$) number of cookies of strength $p(1)>1/2$ at
  each site.
\end{ex}

\begin{ex}[Bounded cookie stacks]
The model above also generalizes the case of finitely many cookies per site. For instance, let the transition matrix $K$ be such that $K_{i,i+1} = 1$ for $i=1,\ldots,N-1$ and $K_{N,N} = 1$, and let the function $p:\mathcal{R}\ra (0,1)$ be such that $p(N) = 1/2$. In this case, if $\eta=\delta_1$ then  $\w_x(j) = p(j)$ if $j\leq N-1$ and $p(j) = 1/2$ if $j\geq N$. 
With a little more thought, one can also obtain random cookie environments that are i.i.d.\ spatially with a bounded number of cookies at each site as in \cite{kzPNERW} subject to the restrictions that there are only finitely many possible values for the cookie strengths $\w_x(j)$ and that $\w_x(j)\in(0,1)$. 
For instance, if 
\[
 K= \begin{pmatrix}
     0 & 1 & 0 & 0 & 0 \\
     0 & 0 & 0 & 0 & 1 \\
     0 & 0 & 0 & 1 & 0 \\
     0 & 0 & 0 & 0 & 1 \\
     0 & 0 & 0 & 0 & 1
    \end{pmatrix}, 
\quad
\mathbf{p} = \begin{pmatrix} p(1) \\ p(2) \\ p(3) \\ p(4) \\ 1/2 \end{pmatrix}, 
\quad\text{and}\quad
\eta = (\a,0,1-\a,0,0), 
\]
then the corresponding cookie environments under the distribution
$\P_\eta$ have 2 cookies at each site, having stack $(p(1),p(2))$ with
probability $\a$ and stack $(p(3),p(4))$ with probability $1-\a$. 
\end{ex}

Before stating our main results, we note two basic facts for ERWs that
are known to hold in much more generality than our model.  We will use
these facts as needed throughout the paper.

\begin{thm}[{\cite[Theorem 1.2]{ABO14} and \cite[Theorem 4.1]{kzERWsurvey}}]\label{th:facts}
\
 \begin{enumerate}
  \item \textbf{Zero-one law for transience.} \label{th:01law}
  For any initial distribution $\eta$ on
  ${\cal R}$ \[P_\eta\left(\lim_{n\to\infty}X_n=\infty\right),\
  P_\eta\left(\lim_{n\to\infty}X_n=-\infty\right)\in\{0,1\}.\]
 \item \textbf{Strong law of large numbers:} \label{th:slln}
   For any initial distribution $\eta$ on ${\cal R}$ there is a
  deterministic $\vp\in[-1,1]$ such
  that\[P_\eta\left(\lim_{n\to\infty}X_n/n=\vp \right)=1.\]
 \end{enumerate}

\end{thm}

\subsection{Main results}

The function $p:\mathcal{R} \ra (0,1)$ used in the construction of the cookie environment $\w$ corresponds to a column vector $\mathbf{p} = (p(1),p(2),\ldots,p(N))^t$ with $i$-th entry given by $p(i)$. 
Note that our assumptions on the Markov chains
$\{\mathbf{R}^x\}_{x\in\Z}$, imply that the limiting average cookie
strength at each site is
\[
 \lim_{k\ra\infty} \frac{1}{k} \sum_{j=1}^k \w_x(j) = \mu \cdot \mathbf{p} =: \bar{p},
\]
and that this limit does not depend on the distribution $\eta$ of $\w_x(1)$. 
It is natural to suspect that the ERW is transient to the right (resp.\ left) when $\bar{p} > 1/2$ (resp.\  $\bar{p}<1/2$). Our first main result below confirms this is the case. Moreover, we also establish that if $\bar{p}\neq 1/2$ the walk has non-zero linear speed and  a Gaussian limiting distribution. 


\begin{thm}\label{th:noncrit}
 Assume that $\bar{p} \neq 1/2$. 
\begin{description}
  \item[Transience] If $\bar{p} > 1/2$ then $P_\eta( \lim_{n\ra\infty} X_n = \infty) = 1$ for any initial distribution $\eta$ on $\mathcal{R}$. Similarly, if $\bar{p} < 1/2$ then $P_\eta( \lim_{n\ra\infty} X_n = -\infty) = 1$ for any $\eta$. 
  \item[Ballisticity] For any initial distribution $\eta$ on
    $\mathcal{R}$ the limiting velocity $\vp$ (see
    Theorem~\ref{th:facts}\ref{th:slln}) is non-zero: it is positive for
    $\bar{p}>1/2$ and negative for $\bar{p}<1/2$.
  \item[Gaussian limit] For any distribution $\eta$ on $\mathcal{R}$ there exists a constant $b = b(K,\mathbf{p},\eta) > 0$ such that 
  \[
   \lim_{n\ra\infty} P_\eta\left( \frac{X_n - n \vp}{b \sqrt{n}} \leq x \right) = \Phi(x) := \int_{-\infty}^x \frac{1}{\sqrt{2\pi}} e^{-z^2/2} \, dz. 
  \]
\end{description}
\end{thm}

The case $\bar{p} = 1/2$ is more interesting (we will refer to this as the critical case). 
In fact, it will follow from our results below that ERWs in the critical case $\bar{p}=1/2$ can exhibit the full range of behaviors that are known for ERWs with boundedly many cookies per site (e.g., transience with sublinear speed and non-Gaussian limiting distributions). 
Before stating our results for the critical case, we
remark that in the case of boundedly many cookies per site
the classification of the long term behavior of the ERW
  is based 
  on a single parameter
  $\d = \E_\eta[ \sum_{j\geq 1} (2\w_x(j) - 1) ]$ which is the
  expected total drift contained in a cookie stack.  Our results below
  depend on two parameters $\d$ and $\tilde\delta$ which are also
  explicit but do not admit such simple interpretation. However, if we
  denote by ${\cal E}^n_0$ the 
number of upcrossings (steps to the right) the ERW makes from $0$ before the $n$-th downcrossing (step to the left) from $0$
 then
    \[\delta=\delta(\eta,K,\mathbf{p}):=\frac{2\rho}{\nu},\quad \text{where } \rho:=\lim\limits_{n\to\infty} (\E_\eta[{\cal E}^n_0] - n) ,\quad  \text{and}\quad \nu:=\lim\limits_{n\to\infty}\frac{\mathrm{Var}_\eta({\cal E}^n_0)}{n}.\] 
    The parameter $\tilde{\delta}$ is defined in a symmetric way using
    downcrossings instead of upcrossings and vice versa so that
    $\tilde{\delta}:=\delta(\eta,K,\mathbf{1}-\mathbf{p})$.  The
    explicit formulas for  the parameters are not
    intuitive and are given in \eqref{parforms} and \eqref{eq:del}.
In the special case of boundedly many cookies per site the parameter $\d$ agrees with the one used in previous papers, $\tilde\d = -\d$, and $\nu=2$. Proposition \ref{deleq} below shows that in general  $\d + \tilde\d = 1-2/\nu$ where $\nu$ need not be equal to $2$ (see Example~\ref{ex:period2}), and so two parameters are needed in the general model.

\begin{thm}\label{th:rectran}
  Fix an arbitrary initial distribution $\eta$ on ${\cal R}$ and let
  $\bar{p} = \mu \cdot \mathbf{p} = 1/2$.  Then, there exist constants
  $\delta$ and $\tilde{\delta}$ satisfying $\delta+\tilde{\delta}<1$
  (see \eqref{parforms} and \eqref{eq:del} for the explicit formulas in terms of $\eta$,
  $K$, and $\bp$) and such that
\begin{itemize}
 \item if $\delta> 1$, then $P_\eta( \lim_{n\ra\infty} X_n = \infty) = 1$; 
 \item if $\tilde{\delta}> 1$, then
   $P_\eta( \lim_{n\ra\infty} X_n = -\infty) = 1$;
 \item if $\delta\leq 1$ and $\tilde{\delta}\leq 1$ then $P_\eta( \liminf_{n\ra\infty} X_n = -\infty, \, \limsup_{n\ra\infty} X_n = \infty) = 1$. 
\end{itemize}
\end{thm}

The next theorem characterizes exactly when the walk is ballistic (i.e., has non-zero limiting velocity). 
\begin{thm}\label{th:LLN}
  Fix an arbitrary initial distribution $\eta$ on ${\cal R}$. Let
  $\bar{p} = 1/2$, $\d$ and $\tilde\d$ be as in Theorem
  \ref{th:rectran}, and $\vp$ be the limiting velocity (see
  Theorem~\ref{th:facts}\ref{th:slln}). Then
 \begin{itemize}
  \item $\vp > 0$ if and only if $\d > 2$. 
  \item $\vp < 0$ if and only if $\tilde\d > 2$. 
 \end{itemize}
\end{thm}

Our final main result concerns the limiting distributions in the
transient cases. We will state the results below for walks that are
transient to the right ($\d > 1$), though obvious symmetry
considerations give similar limiting distributions for walks that are
transient to the left ($\tilde\d > 1$) by replacing $\d$  with $\tilde\d$.  The theorem below not only gives limiting
distributions for the position $X_n$ of the ERW, but
also for the hitting times $T_n$, where for any $x\in \Z$ the hitting
time $T_x = \inf\{ k\geq 0: \, X_k = x \}$.  For $\a \in (0,2)$ and
$b>0$ we will use $L_{\a,b}(\cdot)$ to denote the distribution of the
totally-skewed to the right $\a$-stable distribution with
characteristic exponent
\begin{equation}\label{stablecf}
\log  \int_\R e^{iux} \, L_{\a,b}(dx) = 
 \begin{cases}
  -b |u|^\a \left( 1- i \tan(\frac{\pi\a}{2}) \sgn(u) \right) & \a \neq 1 \\
  -b |u| \left( 1 + \frac{2 i}{\pi} \log |u| \sgn(u) \right) & \a = 1. 
 \end{cases}
\end{equation}
Also, as in Theorem \ref{th:noncrit} we will use $\Phi(\cdot)$ to denote the distribution function of a standard normal random variable. 

\begin{thm}\label{th:limdist}
 Let $\bar{p} = 1/2$ and suppose that $\d>1$. 
\begin{enumerate}
 \item If $\d \in (1,2)$, then there exists a constant $b>0$ such that \label{ldcase1}
 \[
  \lim_{n\ra\infty} P_\eta\left( \frac{T_n}{n^{2/\d}} \leq x \right) = L_{\d/2,b}(x) 
  \quad\text{and}\quad 
  \lim_{n\ra\infty} P_\eta\left( \frac{X_n}{n^{\d/2}} \leq x \right) = 1 - L_{\d/2,b}(x^{-2/\d}). 
 \]
 \item \label{d2LD} If $\d = 2$, then there exist constants $a,b>0$ and sequences $D(n)\sim a^{-1} \log n$ and $\Gamma(n) \sim a n/\log(n)$ such that 
 \[
 \lim_{n\ra\infty} P_\eta\left( \frac{T_n - n D(n)}{n} \leq x \right) = L_{1,b}(x)
   \quad\text{and}\quad 
 \lim_{n\ra\infty} P_\eta\left( \frac{X_n - \Gamma(n)}{n/(\log n)^2} \leq x \right) = 1-L_{1,b}(-x/a^2). 
 \]
 \item If $\d \in (2,4)$, then there exists a constant $b>0$ such that 
 \[
  \lim_{n\ra\infty} P_\eta\left( \frac{T_n - n/\vp}{ n^{2/\d}} \leq x \right) = L_{\d/2,b}(x)
    \quad\text{and}\quad 
  \lim_{n\ra\infty} P_\eta\left( \frac{X_n - n \vp}{ \vp^{1+2/\d} n^{2/\d}} \leq x \right) = 1-L_{\d/2,b}(-x). 
 \]
 \item If $\d = 4$, then there exists a constant $b>0$ such that 
 \[
  \lim_{n\ra\infty} P_\eta\left( \frac{T_n - n/\vp}{b\sqrt{n \log n}} \leq x \right) = \Phi(x)
    \quad\text{and}\quad 
  \lim_{n\ra\infty} P_\eta\left( \frac{X_n - n\vp}{b \vp^{3/2} \sqrt{n \log n}} \leq x \right) = \Phi(x). 
 \]
 \item If $\d > 4$, then there exists a constant $b>0$ such that \label{ldcase5}
 \[
  \lim_{n\ra\infty} P_\eta\left( \frac{T_n - n/\vp}{b \sqrt{n}} \leq x \right) = \Phi(x)
    \quad\text{and}\quad 
  \lim_{n\ra\infty} P_\eta\left( \frac{X_n - n\vp}{b \vp^{3/2} \sqrt{n}} \leq x \right) = \Phi(x). 
 \]
\end{enumerate}
\end{thm} \subsection{An overview of ideas and open questions} Our
approach is based on the following three ingredients which appeared in
the literature 65-20 years ago and were successfully used by other
authors in a number of different contexts: (1) mappings between random
walk paths and branching process trees, (2) diffusion approximations
of branching-like processes (BLPs) and Ray-Knight theorems, and (3) embedded
renewal structures for BLPs.

 \textit{(1) Mappings between random walk paths and
    branching process trees.}
The existence of a bijection between excursions of nearest-neighbor paths on $\Z$ and rooted trees is well known and was observed as early as \cite[Section 6]{Har52}.
In this bijection, if the rooted tree is viewed as a genealogical tree then the offspring of the individuals in the $n$-th generation corresponds to the number of times the nearest-neighbor path on $\Z$ steps from $n$ to $n+1$ in between steps from $n$ to $n-1$. 
Properties of the random walk can then be deduced from properties of the corresponding random trees. 
For instance, the maximum distance of the excursion from $0$ corresponds to the lifetime of the branching process, and the return time of the walk to the origin is equal to twice the total progeny of the branching process over its lifetime.


Similarly, a related bijection is known to exist between rooted trees and nearest-neighbor paths in $\Z$ from $0$ to $n$. 
In this bijection, the rooted tree corresponds to a genealogical tree where there is one immigrant per generation (for the first $n$ generations only) and the offspring of individuals in the $i$-th generation correspond to the number of steps from $n-i$ to $n-i-1$ between steps from $n-i$ to $n-i+1$. 
As with the first bijection, this bijection also can be used to deduce properties of a random walk from properties of the corresponding random tree. 
For instance, the hitting time $T_n = \inf\{k\geq 1: X_k=n\}$ 
is equal to $n$ (the total number of immigrants) plus twice the total number of progeny of the branching process over its lifetime. 

For either of the above bijections, putting a probability measure on random walk paths induces a probability measure on trees and vice versa. 
For instance, for simple symmetric random walks on $\Z$, the first bijection gives a measure on trees corresponding to a critical Galton-Watson process with Geometric($1/2$) offspring distribution, and the second bijection corresponds to a Galton-Watson process with Geometric($1/2$) offspring distribution but with an additional immigrant in the first $n$ generations. 
These bijections were at the core of F.\,Knight's proof of the classical Ray-Knight theorem in \cite{Kni63}. 
For one-dimensional random walks in a random environment (RWRE), the corresponding measures on trees are instead branching processes with random offspring distributions amd the second bijection above was used in \cite{kksStable} to obtain limiting distributions for transient one-dimensional RWRE (under the averaged measure).  
The advantage of using this bijection to study the RWRE is that while the random walk is non-Markovian under the averaged measure, the corresponding branching process with random offspring distribution is a Markov chain. 
Subsequently, it has also been observed that several other models of self-interacting random walks also have this Markovian property for the BLPs which come from the above bijections.
In particular, this applies to a large number of self-interacting random walk models where the transition probabilities for the walk depend on the past behavior of the walk at the present site. Included in this are several models of self-repelling/attracting random walks \cite{tTSAW,tGRK,pHCEI,ptRWSBFE} as well as excited random walks with boundedly many cookies per site. 
For ERWs, this branching process approach was first used in \cite{bsCRWspeed} and has since been the basis for many of the subsequent results on ERWs; see for instance the review \cite{kzERWsurvey} as well as more recent works in \cite{kzEERW,kosERWPC,aoEM,pXSDERW}.

\textit{(2) Diffusion approximations of BLPs and Ray-Knight
  theorems.}
Diffusion approximations of branching processes is an extensively studied topic, especially in the context of applications to population dynamics (see, for example, books \cite{Jag75}, \cite{Kur81}). 
The fact that a rescaled critical Galton-Watson process converges weakly to a diffusion was first observed in \cite{Fel51} and then rigorously proven in \cite{Jir69}. In particular, a Galton-Watson
process with Geometric$(1/2)$ offspring distribution converges to
one-half of a zero-dimensional squared Bessel process. Adding an
immigrant in each generation raises the above dimension to two so that
the limiting process becomes one-half of the square of a standard 
two-dimensional Brownian motion. 
Recalling the connection of simple symmetric random walks with critical Galton-Watson processes from the bijections given above and the fact that Brownian motion is the scaling limit of random walks, we can think about the first and second classical Ray-Knight Theorems as continuous versions of these bijections.\footnote{The second Ray-Knight
  theorem was extended from Brownian motion to a large class of
  symmetric (and some non-symmetric) Markov processes
  (\cite{EKMRS}). These extensions found many new applications. The
  interested reader is referred to books \cite{MR06}, \cite{Sz12} and
  references therein.}

For a class of non-Markovian self-interacting random motions, a
generalized Ray-Knight theory was developed by B\'alint T\'oth (see
\cite{Tot99} and references therein). For these walks the corresponding diffusion
process limits for the local times are squared Bessel processes of appropriate dimensions. Using this generalized Ray-Knight theory the author obtains limiting distributions for the random walk stopped at an independent geometric time.
Additionally, for a certain sub-class of these self-interacting random walks he identifies these limiting distributions as those of a Brownian motion perturbed at extrema \cite[Remark on p.\,1334]{tGRK} and asks if the
observed connection extends to multi-dimensional distributions. 

Even though the dynamics of ERWs are quite different from those of the
self-interacting random walks in \cite{Tot99}, the corresponding
rescaled BLPs also converge to squared Bessel
processes (see \cite{kmLLCRW}, \cite{kzEERW}, \cite{DK14}, and
Lemmas~\ref{lem:DA} and \ref{al} below). 
Similarly to \cite{tGRK} these diffusion limits can be used to infer certain properties of the scaling behavior of the ERW, though alone they are not quite sufficient to obtain limiting distributions for the ERWs.

\textit{(3) Embedded renewal structure.}
To obtain limit theorems for transient ERWs we essentially follow an outline which was first presented in
\cite[pp.\,148-150]{kksStable} in the context of transient one-dimensional random walks in random environments. Starting with \cite{bsRGCRW}, 
 this strategy has been used in essentially all papers concerned with limit laws for transient one-dimensional ERW.
The idea is to first use the bijection above which relates the hitting time $T_n$ to $n$ plus twice the total progeny of a BLP, and then to compare the total progeny of this BLP to a sum of i.i.d.\ random variables using regeneration times of the BLP. To obtain limiting distributions for the hitting times (and then also the position) of the ERW using this approach, the key is to obtain precise tail asymptotics of both the regeneration times and the total progeny between regeneration times of the BLPs.

In \cite{bsRGCRW} and \cite{kzPNERW}, the necessary tail asymptotics
for the BLPs were obtained using generating functions or modifications
of the process which allowed for the application of known results from
the literature for branching processes with
migration. The approach based on squared Bessel diffusion
  limits of the BLPs and calculations in the spirit of gambler's ruin
  was proposed in \cite{kmLLCRW} and developed in subsequent papers
  \cite{kzEERW}, \cite{DK14}. This approach not only eliminates the
  need to quote results from the branching process literature but also
  allows one to obtain new results about BLPs.
In the current work we push
the method further by lifting the limitation on the number of cookies
per site at the cost of requiring a Markovian structure within the
cookie stacks.  We provide the required tail asymptotics in the full
critical regime (Theorems~\ref{th:fbp} and \ref{th:bbp}), show the
one-dimensional limit laws in the transient case, but leave the
recurrent case and functional limit theorems for future work.

\medskip

\subsubsection{Comments and open questions } 
\ 

(i) Using the formulas in \eqref{parforms} and
\eqref{eq:del} below, one can explicitly calculate the parameters $\d$
and $\tilde{\delta}$ in terms of $\eta$, $K$, and $\mathbf{p}$.
Therefore, recurrence/transience, ballisticity, and the type of the
limiting distribution can be determined for any given example.
However, there is no explicit formula for the limiting velocity $\vp$
when $\vp\neq 0$ or for the scaling parameters $a,b>0$ that appear in
Theorem \ref{th:limdist}.
\begin{op}
  What can be said about monotonicity and strict monotonicity of $\vp$
  with respect to the cookie environment?\footnote{See \cite{pCRWMnt}
    and \cite{Hol15} for the up to date account of the known results.}
\end{op}

(ii) Theorems \ref{th:LLN} and \ref{th:limdist} give the first such
results for ERWs with an unbounded number of cookies
per site (with the exception of the special case of random walks in
random environments).  As mentioned above, the recurrence/transience
results in Theorem \ref{th:rectran} were known for excited random
walks with unbounded number of cookies per site only in the special
cases of non-negative cookie drifts ($\w_x(j) \geq 1/2$ for all $j$)
\cite{zMERW} or periodic cookie stacks \cite{kosERWPC}.  In
\cite{kosERWPC} the authors also use a BLP, but their proof of the
criterion for recurrence/transience differs from ours and is based on
the construction of appropriate Lyapunov functions rather than on the
analysis of the extinction
times of the BLP.  It is possible
  that their proof may be extended to include our model, but such an
  extension is not automatic since it requires an additional strong
  concentration estimate (see \cite[Theorem 1.3]{kosERWPC}) while our
  approach seems much less demanding.


(iii) Functional limit theorems have been obtained for ERWs
    with bounded cookie stacks under the i.i.d.\ and (weak)
    ellipticity assumptions. The transient case was handled in
    \cite[Theorem 3]{kzPNERW} and \cite[Theorems 6.6 and
    6.7]{kzERWsurvey}. Scaling limits for the recurrent case were
    obtained in \cite{dCLTERW} and \cite{dkSLRERW}. Excursions from
    the origin and occupation times of the left and right semi-axes
    for ERW with bounded number of cookies per site have also been
    studied \cite{kzEERW}, \cite{DK14}. We believe that similar
    results hold for the model considered in the current paper and
    leave this study for future work.
    \begin{op}
      State and prove functional limit theorems for the critical transient case
      ($\bar{p}=1/2$, $\max\{\delta,\tilde{\delta}\}>1$).
    \end{op}
\begin{op}
  Study scaling limits and occupation times of the right and left
  semi-axes for the recurrent case ($\bar{p} = 1/2$,
  $\max\{\d,\tilde\d\} \le 1$).
\end{op}
\begin{op}
  Show that the critical ERW (i.e.\ $\bar{p}= 1/2$) is strongly
  transient\footnote{$(X_n)_{n\ge 0}$ is said to be strongly transient
    under $P_\eta$ if it is transient and
    $E_\eta[R\,|\,R<\infty]<\infty$, where
    $R:=\inf\{n\ge 1:\,X_n=X_0\}$.} under $P_\eta$ if and only if
  $\max\{\delta,\tilde{\delta}\}>3$. Show that the non-critical ERW
  (i.e.\ $p\ne 1/2$) is always strongly transient.
    \end{op}

\subsection{Examples}
In this subsection we give some examples where the parameters $\d$ and $\tilde\d$ are explicitly calculated using the formulas in \eqref{parforms} and \eqref{eq:del}. 
The calculations are somewhat tedious to do by hand, but since the formulas are explicit one can usually compute the parameters very quickly with technology. 

\begin{ex}[Periodic cookie stacks]\label{ex:period2}
In the setting of Example~\ref{ex:period1},
\[
 \d = \frac{\sum_{j=1}^N \sum_{i=1}^j (1-p(j))(2p(i)-1)}{2\sum_{j=1}^N p(j)(1-p(j))} 
\quad\text{and}\quad
\ \tilde\d = \frac{\sum_{j=1}^N \sum_{i=1}^j p(j)(1-2p(i))}{2\sum_{j=1}^N p(j)(1-p(j))}
\]
For this model, the characterization of recurrence and transience was proved previously in \cite{kosERWPC} but the results in Theorems \ref{th:LLN} and \ref{th:limdist} are new. 
\end{ex}

\begin{ex}[Cookie stacks of geometric height]\label{ex:geo}
In the setting of Example~\ref{geom} one obtains that
$\d = -\tilde{\delta}=(2p(1)-1)/\alpha$. 
Note that $\E_1[ \sum_{j=1}^\infty (2 \w_0(j) - 1)] = (2p(1)-1)/\a$ as well; 
this agrees with the criteria for recurrence/transience for this example that can be obtained from \cite{zMERW}. Previously, there were no known results regarding ballisticity or limiting distributions for this example. 

\begin{figure}[h]
 \includegraphics[scale=0.5]{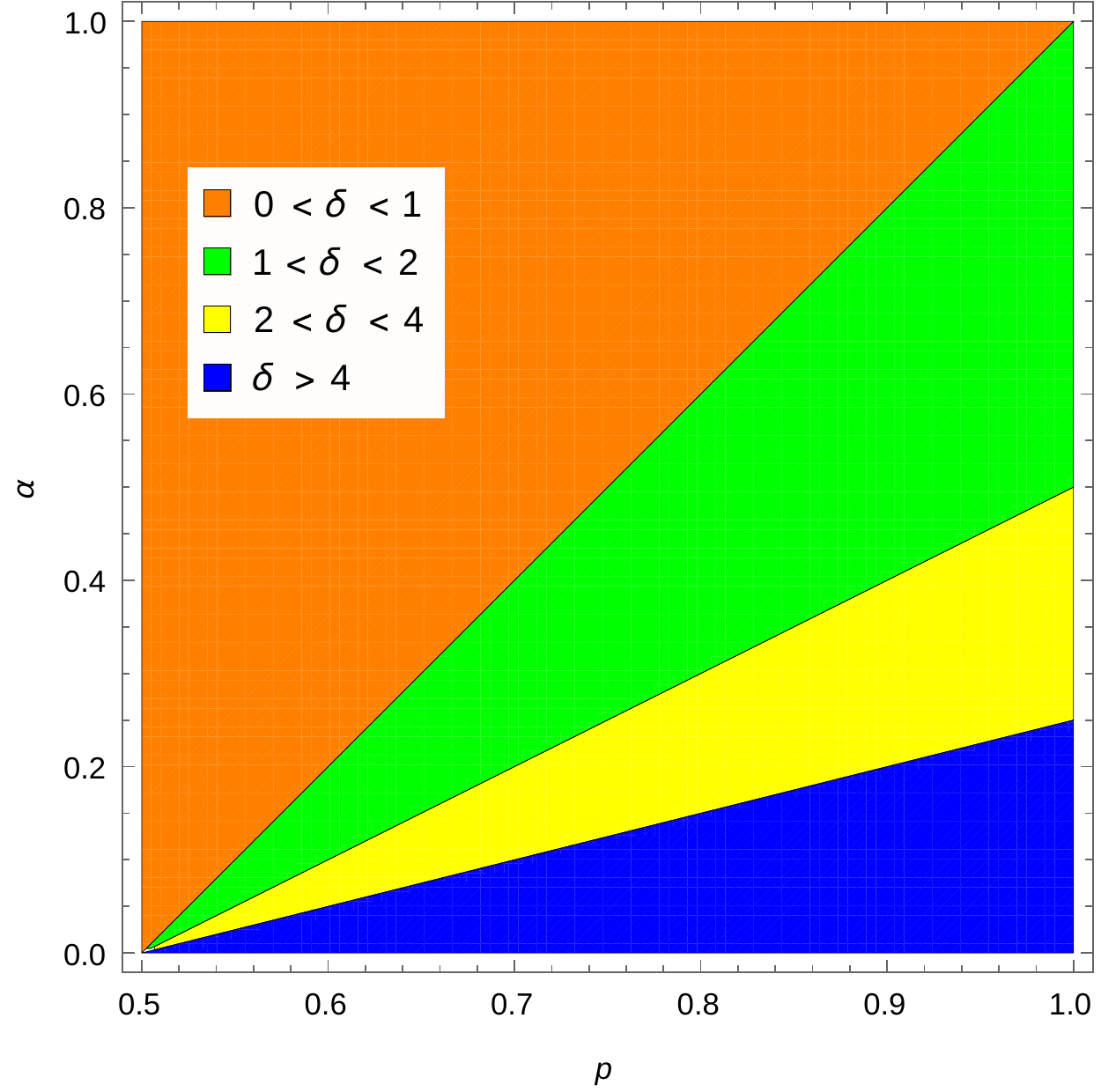}
\caption{A plot showing the regions for different types of behavior occurring for cookie environments as in Example \ref{ex:geo}.
} \label{fig1}
\end{figure}
\end{ex}

\begin{ex}[Two-type, critical]\label{ex:2tb}
 Let $K = \begin{pmatrix} 1-\a & \a \\ \a & 1-\a \end{pmatrix}$ for some $\a \in (0,1)$ and let $\mathbf{p} = (p,1-p)^t$ for some $p>1/2$. In this case, if we use the initial condition $\eta = (1,0)$ then a calculation done with Mathematica yields 
\begin{equation}\label{eq:th2tb}
 \d = \d\left( \begin{pmatrix} 1-\a & \a \\ \a & 1-\a \end{pmatrix}, \begin{pmatrix} p \\ 1-p \end{pmatrix}, (1,0) \right) = \frac{(2 p-1) ((2 \alpha -1) p-\alpha )}{4 (2 \alpha -1) (p-1) p + \a -1}.
\end{equation}
Note that if $\a \in (0,1/4)$ then the parameter $\d$ is non-monotone in $p$. 
In fact, for $\a < \frac{14 - 3 \sqrt{21}}{56} \approx 0.00450487$ the parameter $\d$ starts at 0, increases to a value larger than $4$, and then decreases to $1$ as $p$ ranges from $1/2$ to $1$ (See Figure \ref{fig:2typebalanced}). 
\begin{figure}
\hfill
  \includegraphics[scale=0.5]{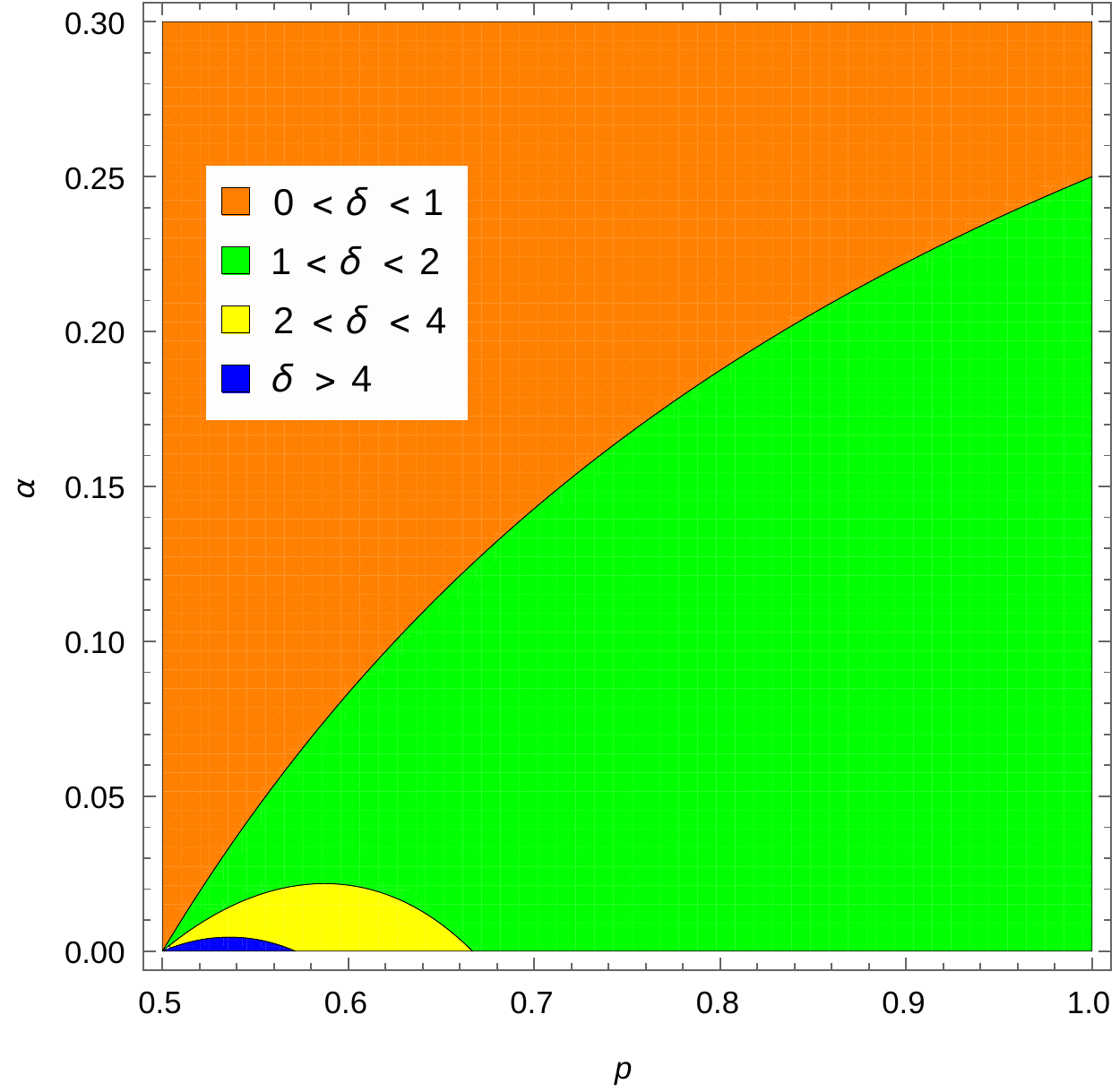}
\hfill
 \includegraphics[scale=0.6]{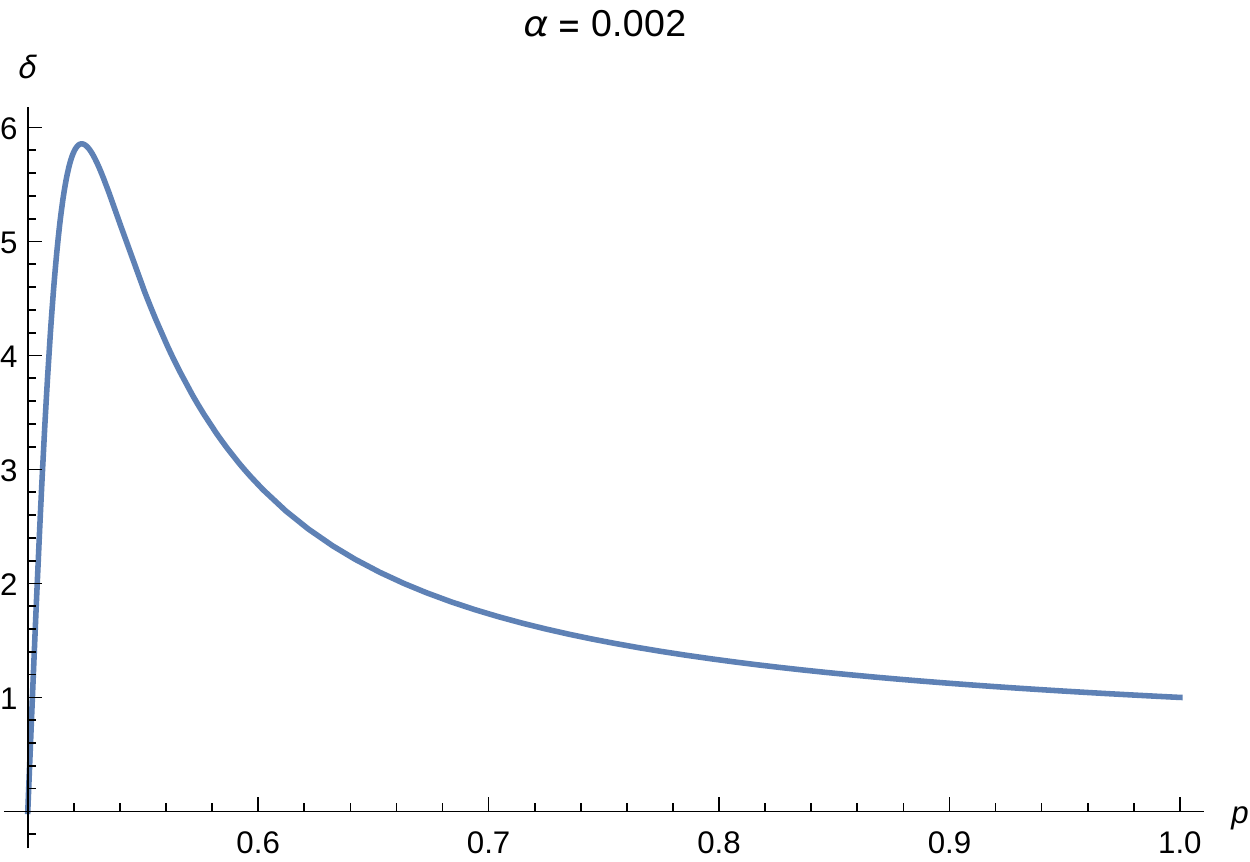}
\hfill
\caption{The left plot shows the regions for different types of
  behavior occurring for cookie environments as in Example
  \ref{ex:2tb}. 
  The right plot shows the value of $\d$ as a function of
  $p \in (1/2,1)$ when $\a=0.002$.} \label{fig:2typebalanced}
\end{figure}
\end{ex}
\noindent
Note also that the case $\a = 0$ corresponds to a classical simple
random walk which steps to the right with probability $p>1/2$ on each
step. However, taking $\a = 0$ in \eqref{eq:th2tb} gives
$\d = \frac{p}{2p-1}$ and thus the results of Theorems \ref{th:LLN}
and \ref{th:limdist} clearly do not hold when $\a=0$ (this is not a
contradiction since the transition matrix $K$ is the identity matrix
when $\a=0$ and there is more than one irreducible closed
  set). 
The case $\a=1$ gives periodic cookie stacks with period $2$ and
the formula \eqref{eq:th2tb} gives $\d = \frac{(2p-1)(1-p)}{4p(1-p)}$
which agrees with the formula obtained in
  \cite{kosERWPC}, see Example~\ref{ex:period2}.

\subsection{Structure of the paper} In Section~\ref{sec:branching} we
introduce our main tool, forward and backward 
BLP,  and show how the main results in the critical case (Theorems~\ref{th:rectran} - \ref{th:limdist}) can be deduced from
Theorems~\ref{th:fbp} and \ref{th:bbp} about the behavior of the tails
of the lifetime and the total progeny over a lifetime of these BLP.
Much of the remainder of the paper therefore is focused on the analysis of these BLPs, with the goal of proving Theorems \ref{th:fbp} and \ref{th:bbp}. 
Section~\ref{sec:meanvar} discusses the asymptotics of the
mean and variance of the forward BLP. Section~\ref{sec:param} extends
the results of the previous section to the backward BLP and
establishes key relationships between parameters of the BLPs; in particular, in Section \ref{sec:param} we give explicit formulas for $\delta$ and $\tilde{\delta}$. In
Section~\ref{sec:noncrit} we treat the non-critical case
($\bar{p}\ne 1/2$) deriving Theorem~\ref{th:noncrit}. 
Using the formulas for the parameters of the BLPs derived in the previous sections, we show how the BLPs in the non-critical case $\bar{p}>1/2$  can be coupled with critical BLPs for which the parameter $\d$ is arbitrarily large. 
From this coupling we then show that the conclusions of Theorem \ref{th:noncrit} can be derived from Theorems \ref{th:fbp} and \ref{th:bbp} in the same way as the proofs of Theorems \ref{th:rectran}--\ref{th:limdist} were obtained when $\delta > 4$. 
Finally, in Sections~\ref{sec:3lem} and
\ref{sec:tails} we discuss proofs of Theorems ~\ref{th:fbp} and
\ref{th:bbp}. 
The crucial step here is showing that the BLPs have scaling limits that are squared Bessel processes. The generalized dimensions of these squared  Bessel processes depend on the asymptotics of the mean and variance of the BLPs computed in Sections \ref{sec:meanvar} and \ref{sec:param}. 
This provides the connection of the parameters $\d$ and $\tilde\d$ defined in Section \ref{sec:param} with the tail asymptotics exponents in Theorems \ref{th:fbp} and \ref{th:bbp}. 


\section{The associated branching-like processes}\label{sec:branching}
In this section, we introduce two BLPs that are naturally associated
with the ERW and will prove the main theorems for the critical case
assuming the necessary results about these BLP. 


Given a cookie environment $\w$, we will expand the measure $P_\w$ to
include an independent family of Bernoulli random variables
$\{\xi_x(j)\}_{x\in\Z, \, j\geq 1}$ such that
$\xi_x(j) \sim \text{Bernoulli}(\w_x(j))$.  The ERW
can then be constructed from the $\xi_x(j)$ as follows: if $X_n = x$
and $\sum_{k=0}^n \ind{X_k = x} = j$, then
$X_{n+1} = X_n + 2\xi_x(j) - 1$. 
  \begin{rem}
  The arrow systems construction of the ERW
  introduced by Holmes and Salisbury \cite{HS12} is very similar to
  the above coin toss construction: compare the last equation with
  \cite[p.\,2, above Theorem 1.1]{Hol15}. The distinction lies in the
  emphasis of the arrow approach on combinatorial results obtained by
  coupling of arrow systems rather than probability measures.
  \end{rem}

We now show how the family of Bernoulli random variables $\{\xi_x(j)\}$ can also be used to construct the associated forward and backward BLP.

\subsection{The forward branching-like process}\label{ssec:fblp}
The excursions of a random walk to the right of the origin induce a natural tree-like structure on the right-directed edge local times of the walk. 
That is, for $x\geq 1$ jumps from $x$ to $x+1$ can be thought of ``descendants'' of previous jumps from $x-1$ to $x$. 
To be precise, if we set $\gamma_0 = 0$ and for $n\geq 1$ and $x\geq 0$ let
\[
 \gamma_n = \inf\{k > \gamma_{n-1}: \, X_{k-1} > 0 \text{ and } X_k = 0 \}
\quad \text{and}\quad 
 \mathcal{E}_x^n = \sum_{k=0}^{\gamma_n-1} \ind{X_k = x, \, X_{k+1} =x+1},
\]
then $\gamma_n$ is the time of the $n$-th return to the origin from the right and $\mathcal{E}_x^n$ is the number of times the random walk has traversed the directed edge from $x$ to $x+1$ by time $\gamma_n$ (note that $\gamma_n$ can be infinite if the walk is transient to the right or left).

If the random walk makes $n$ excursions to the right ($\gamma_n < \infty$), then the directed edge local times $\mathcal{E}_x^n$ can be computed from the Bernoulli random variables. If $\mathcal{E}_{x-1}^n = m$, then since the steps of the random walk are $\pm 1$ it follows that the walk makes $m$ jumps to the left from $x$ by time $\gamma_ n$. Thus, $\mathcal{E}_x^n$ is the number of jumps that the random walk makes to the right from $x$ before the $m$-th jump to the left. If we refer to a Bernoulli random variable $\xi_x(j)$ as a ``success'' if $\xi_x(j) = 1$ and a ``failure'' if $\xi_x(j) = 0$, then $\mathcal{E}_x^n$ is the number of successes in the Bernoulli sequence $\xi_x = \{\xi_x(j)\}_{j\geq 1}$ before the $m$-th failure.  
More precisely, introducing the notation
\[
 S_m^x = \inf\left\{ k \geq 0: \sum_{j=1}^{k+m} (1-\xi_x(j)) = m \right\},
\]
we have that if $\gamma_n<\infty$
 and $\mathcal{E}_{x-1}^n = m$ then $\mathcal{E}_x^n = S_m^x$. 
Note that $S_0^x = 0$ by the convention that an empty sum is equal to zero. 
Also, let $G_m^x = S_m^x - S_{m-1}^x$ so that $G_m^x$ is the number of successes between the $(m-1)$-st failure and the $m$-th failure in the Bernoulli sequence $\xi_x$. 

With the above directed edge process $\mathcal{E}_x^n$ as motivation, we define the forward BLP started at $y\geq 1$ by 
\begin{equation}\label{fbpdef}
 U_0 = y \quad\text{and}\quad U_i = S_{U_{i-1}}^i = \sum_{m=1}^{U_{i-1}} G_m^i \quad \text{for } i \geq 1. 
\end{equation}
We will use the notation $P_\w^{U,y}(\cdot)$ and
$P_\eta^{U,y}(\cdot) = \E_\eta[ P_\w^{U,y}(\cdot)]$ for the quenched
and averaged distributions, respectively, of the forward BLP started
at $y$.  Note that under the quenched measure $P_\w^{U,y}$ the
forward BLP is a time-inhomogeneous Markov chain, but since the cookie
environment is assumed to be spatially i.i.d.\ the forward BLP is a
time-homogeneous Markov chain under the averaged measure
$P_\eta^{U,y}$ for any initial distribution $\eta$.

The following Lemma summarizes the connection of the forward BLP with the directed edge local times of the random walk. 
\begin{lem}\label{lem:fbpExn}
  If $U_0 = n\geq 1$, then $U_i \geq \mathcal{E}_i^n$ for all
  $i\geq 0$. Moreover, on the event $\{\gamma_n < \infty\}$ we have
  $U_i = \mathcal{E}_i^n$ for all $i\geq 0$ if $U_0 = n$.
\end{lem}
The equality $U_i = \mathcal{E}_i^n$ on the event
$\{\gamma_n < \infty\}$ was described above. For the proof of the
general inequality $U_i \geq \mathcal{E}_i^n$ we refer the reader to
\cite[Section 4]{kzPNERW} or \cite[Lemma 2.1]{pCRWMnt}.

\begin{rem}
 Note that in the case of classical simple random walks (i.e., $\w_x(j) \equiv p \in (0,1)$) $\{G_m^i\}_{i\in\Z, \, m\geq 1}$ is a sequence of i.i.d.\ Geometric($1-p$) random variables. In this case, it is clear from \eqref{fbpdef} that $U_i$ is a branching process with Geometric($1-p$) offspring distribution. For ERWs with boundedly many cookies per site the $G_m^i$ are i.i.d.\ Geometric($1/2$) for all $m$ sufficiently large and thus one can interpret $U_i$ as a branching process with (random) migration (c.f. \cite{kzPNERW}, or more explicitly but in the context of the backward BLP see \cite{bsCRWspeed}). However, in the more general setup of the current paper one can no longer interpret $U_i$ as a branching process and so we simply refer to $U_i$ as a ``branching-like'' process. 
\end{rem}

\begin{rem}\label{rem:lfbp}
 One can also obtain a similar BLP  which is related to the excursions of the ERW to the left of the origin. Clearly this BLP would have the same law as the forward BLP $U_i$ defined here but with $\mathbf{p}$ replaced by $\one - \mathbf{p}$. 
\end{rem}

\subsection{The backward branching-like process}\label{ssec:bblp}
The backward BLP is related to the random walk through edge local times. 
However, the backward BLP is related to the local times of the \emph{left} directed edges when the random walk first reaches a fixed point to the right of the origin. 
To be precise, 
\[
\mathcal{D}_n^x = \sum_{k=0}^{T_n-1} \ind{X_k = x, \, X_{k+1} = x-1}, \quad n\geq 1, \, x < n, 
\]
be the number of steps to the left from $x$ before time $T_n$ (recall that $T_n = \inf\{k\geq 0: X_k = n\}$ is the hitting time of $n$ by the ERW). 
On the event $\{T_n < \infty\}$, the sequence of directed edge local times $\{\mathcal{D}_n^x\}_{x\leq n}$ also have a branching-like structure. 
Jumps to the left from $x+1$ before time $T_n$ give rise to subsequent jumps to the left from $x$ before time $T_n$. However, one important difference with the forward BLP should be noted in that not all jumps to the left from $x$ are ``descendants'' of jumps to the left from $x+1$. In particular, for $x\geq 0$ the random walk can jump to the left from $x$ before ever jumping from $x$ to $x+1$.

As with $\mathcal{E}_x^n$ above, the directed edge process $\mathcal{D}_n^x$ can be computed from the Bernoulli random variables $\xi_x(j)$.  
In particular, if 
\[
 F_m^x = \inf\left\{ k \geq 0: \sum_{j=1}^{k+m} \xi_x(j) = m \right\}
\]
denotes the number of failures in the Bernoulli sequence $\xi_x = \{\xi_x(j)\}_{j\geq 1}$ before the $m$-th success 
then it is easy to see that if $T_n < \infty$ then $\mathcal{D}_n^n = 0$ and 
\begin{equation}\label{Dxn}
 \text{if } \mathcal{D}_n^{x+1} = m \quad\text{then}\quad \mathcal{D}_n^x = 
\begin{cases} 
 F_{m+1}^x & x\geq 0 \\
F_m^x & x < 0, 
\end{cases}
\quad \text{for } x < n. 
\end{equation}
Indeed, if $x\geq 0$ and there are $m$ jumps from $x+1$ to $x$ before time $T_n$ then there must be $m+1$ jumps from $x$ to $x+1$ (the initial jump from $x$ to $x+1$ plus $m$ more jumps which can be paired with a prior jump from $x+1$ to $x$). Thus, from the construction of the random walk via the Bernoulli random variables $\xi_x(j)$ above, it follows that the number of jumps from $x$ to $x-1$ before time $T_n$ is the number of failures before the $(m+1)$-th success in the Bernoulli sequence $\xi_x$. 
The explanation of \eqref{Dxn} when $x<0$ is similar, with the exception that all jumps to the right from $x$ can be paired with a prior jump to the left from $x+1$. 

Again with the directed edge local time process as motivation, we define the backward BLP started at $y\geq 0$ by 
\begin{equation}\label{bbpdef}
 V_0 = y \quad\text{and}\quad V_i = F_{V_{i-1}+1}^i \quad \text{for } i\geq 1. 
\end{equation}
We will use $P_\w^{V,y}$ and $P_\eta^{V,y}$ to denote the quenched and averaged laws of the backward branching process started at $y\geq 0$. As with the forward BLP, $V_i$ is a time-inhomogeneous Markov chain under the quenched measure and a time-homogeneous Markov chain under the averaged measure. 
\begin{lem}\label{lem:bbpDxn}
 If $n\geq 1$ and $P_\eta( T_n < \infty) = 1$, then the sequence $(\mathcal{D}_n^n,\mathcal{D}_n^{n-1},\ldots,\mathcal{D}_n^1,\mathcal{D}_n^0)$ has the same distribution under $P_\eta$ as the sequence $(V_0,V_1,\ldots,V_{n-1},V_n)$ under the measure $P_\eta^{V,0}$. 
\end{lem}

\begin{proof}
  For $n\geq 1$ let $\{ \xi^{(n)}_x(j) \}_{x \in \Z, \, j\geq 1}$ be
  the family of Bernoulli random variables given by
  $\xi_x^{(n)}(j) = \xi_{n-x}(j)$ for $x\in \Z$ and $j\geq 1$, and let
  $V_i^{(n)}$ be the backward BLP started at $V_0^{(n)} = 0$ but
  defined using the Bernoulli family $\{\xi_x^{(n)}(j) \}$ in place of
  $\{\xi_x(j) \}$.  Then, it is clear from \eqref{Dxn} and
  \eqref{bbpdef} on the event $\{T_n < \infty \}$ that
  $\mathcal{D}_n^{n-i} = V_{i}^{(n)}$ for $i=0,1,\ldots,n$.  Finally,
  since the cookie environments are spatially i.i.d., it follows that
  $\{\xi_x^{(n)}(j)\}_{x\in \Z,\, j\geq 1}$ and
  $\{\xi_x(j)\}_{x\in \Z,\, j\geq 1}$ have the same distribution under
  the averaged measure $P_\eta = \E_\eta[ P_\w(\cdot) ]$.
\end{proof}

\subsection{Proofs of the main results in the critical case}\label{ssec:proofmain}

Here and throughout the remainder of the paper, we will use the following notation for hitting times of stochastic processes. If $\{Z_j\}_{j\geq 0}$ is a stochastic process, then for any $x\in\R$ let $\s_x^Z$ and $\tau_x^Z$ be the hitting times 
\begin{equation*}
 \sigma_x^Z:=\inf\{j > 0:\,Z_j\le x\}\ \ \text{and}\ \ \tau_x^Z:=\inf\{j\ge 0:\ Z_j\ge x\}.
\end{equation*}
Usually the stochastic process $Z$ will be either the forward or backward BLP, though occasionally we will also use this notation for other processes.

The analysis of the forward and backward BLP is key to the proofs of all the main results in this paper. 
In particular, all of the results in the critical case $\bar{p} = 1/2$ (Theorems \ref{th:rectran}--\ref{th:limdist}) will follow from the following two Theorems.  
\begin{thm}
  \label{th:fbp} Let $\bar{p}=1/2$ and let $\delta$ be given by (\ref{eq:del}). 
If $U$ is the forward BLP defined in (\ref{fbpdef}), then 
\begin{itemize}
 \item If $\d>1$ then $P_\eta^{U,y}(\sigma_0^U=\infty)>0$ for all $y\geq 1$. 
 \item If $\d\leq 1$ then for any $y \geq 1$ there are positive constants $\Cl[c]{fl}=\Cr{fl}(y,\eta)$ and $\Cl[c]{fp}=\Cr{fp}(y,\eta)$
  such that
\begin{equation}
   \label{fbptails}
   \lim_{n\to\infty}n^{1-\delta} P_\eta^{U,y}(\sigma_0^U>n)=\Cr{fl}\ \ \text{and}\ \ \lim_{n\to\infty}n^{(1-\delta)/2} P_\eta^{U,y}\left(\sum_{j=0}^{\sigma_0^U-1}U_j>n \right)=\Cr{fp},
 \end{equation}
where for $\delta=1$ we replace $n^0$ with $\ln n$.
\end{itemize}
\end{thm} 

\begin{thm}
  \label{th:bbp} Let $\bar{p}=1/2$ and let $\delta$ be given by (\ref{eq:del}). 
  If $V$ is the backward BLP defined in (\ref{bbpdef}), then 
\begin{itemize}
 \item If $\delta<0$ then $P_\eta^{V,y}(\sigma_0^V=\infty)>0$ for all $y\geq 0$. 
 \item If $\d\geq 0$ then 
%
for any $y \geq 0$ there are positive constants $\Cl[c]{bl}=\Cr{bl}(y,\eta)$ and $\Cl[c]{bp}=\Cr{bp}(y,\eta)$
  such that
\begin{equation}
   \label{bbptails}
   \lim_{n\to\infty}n^\delta P_\eta^{V,y}(\sigma_0^V>n)=\Cr{bl}\ \ \text{and}\ \ \lim_{n\to\infty}n^{\delta/2} P_\eta^{V,y}\left(\sum_{j=0}^{\sigma_0^V-1}V_j>n \right)=\Cr{bp},
 \end{equation}
where for $\delta=0$ we replace $n^0$ with $\ln n$.
\end{itemize}
\end{thm} 
\begin{rem}
Analogous theorems are known for the ERW with bounded number of cookies per site, \cite[Theorems 2.1, 2.2]{kmLLCRW}, \cite[Theorem 1.3]{DK14}. 
\end{rem}

We will give the proofs of Theorems \ref{th:fbp} and \ref{th:bbp} in Sections \ref{sec:3lem} and \ref{sec:tails} below, 
but first we will show how they are used to prove Theorems \ref{th:rectran}--\ref{th:limdist}. 



\begin{proof}[Proof of Theorem \ref{th:rectran}]
For any cookie environment $\w$, let $\w^+$ be the modified cookie environment in which $\w^+_x(j) = \w_x(j)$ for $x\neq 0$, $j\geq 1$ and with $\w^+_0(j) = 1$ for all $j\geq 1$ (that is, in the cookie environment $\w^+$ the walk steps to the right after every visit to the origin). 
Recall that $\gamma_n$ is the time of the $n$-th return to the origin. Lemma \ref{lem:fbpExn} implies that 
\begin{equation}\label{gnfin}
 P_{\w^+}( \gamma_n < \infty ) = P_{\w^+}( \mathcal{E}_n^x = 0, \text{ for all $x$ sufficiently large} )
\leq P_{\w}^{U,n}( \s_0^U < \infty ). 
\end{equation}
(Note that in the last probability on the right we can change the cookie environment from $\w^+$ to $\w$ since the forward branching process is generated using the Bernoulli random variables $\xi_x(j)$ with $x\geq 1$.)
If instead $\gamma_n = \infty$, then since the $\w_x(j)$ are uniformly bounded away from $0$ and $1$ the walk cannot stay bounded for the first $n$ excursions to the right. Therefore, 
\begin{equation}\label{gninf}
  P_{\w^+}( \gamma_n = \infty ) = P_{\w^+}( \mathcal{E}_n^x \geq 1, \, \forall x\geq 1)
\leq 
P_\w^{U,n}(\s_0^U = \infty),
\end{equation}
where the last inequality follows from Lemma \ref{lem:fbpExn}. Combining \eqref{gnfin} and \eqref{gninf} we can conclude that 
\begin{equation}\label{gns0}
 P_{\w^+}( \gamma_n < \infty ) = P_{\w}^{U,n}( \s_0^U < \infty ), \quad \forall n\geq 1. 
\end{equation}

Suppose that $\d \leq 1$. 
Then it follows from \eqref{gns0} and Theorem \ref{th:fbp} that $E_\eta[ P_{\w^+}( \gamma_n < \infty ) ] = 1$ for all $n\geq 1$. That is, with probability one every excursion of the ERW to the right of the origin will eventually return to the origin. 
Similarly, if $\tilde\d \leq 1$ then all excursions to the left of the origin eventually return to the origin. 
Therefore, if $\d,\tilde\d \leq 1$ then all excursions from the origin are finite and so the walk returns to the origin infinitely many times. 
Since all $\w_x(j)$ are uniformly bounded away from $0$ and $1$ this implies that the walk visits every site infinitely often. 

If instead $\d>1$, then \eqref{gns0} and Theorem \ref{th:fbp} imply that $E_\eta[ P_{\w^+}( \gamma_1 = \infty ) ] > 0$, and thus
\begin{align*}
 0 < E_\eta[\w_0(1)] E_\eta[ P_{\w^+}( \gamma_1 = \infty ) ] 
&\leq P_\eta( X_n \geq 1, \, n\geq 1 ) 
\leq P_\eta\left( \lim_{n\ra\infty} X_n = \infty \right), 
\end{align*}
where the last inequality again follows from the fact that the $\w_x(j)$ are uniformly bounded away from $0$ and $1$. 
Finally, we can conclude by Theorem \ref{th:facts} that $\d>1$ implies that the walk is transient to the right with probability $1$. 
A similar argument shows that $\tilde\d>1$ implies that the walk is transient to the left. 
\end{proof}

\begin{proof}[Proof of Theorem \ref{th:LLN}]
Since the limiting speed $\vp$ exists by Theorem \ref{th:facts}\ref{th:slln}, if the walk is recurrent then the speed must be $v_0 = 0$. Thus, we need only to consider the case where the walk is transient. 
First, assume that the walk is transient to the right ($\d>1$). 
Since $T_n/n = T_n/X_{T_n}$, the existence of the limiting speed $\vp$ implies that $\lim_{n\ra\infty} T_n/n = 1/\vp$, where we use the convention $1/0 = \infty$. 
It is easy to see that $T_n = n + 2 \sum_{x<n} \mathcal{D}_n^x$ for every $n\geq 1$. Since the walk is transient to the right, 
\begin{equation}\label{Dxnneg}
 \lim_{n\ra\infty} 2 \sum_{x<0} \mathcal{D}_n^x \leq \sum_{k=0}^\infty \ind{X_k \leq -1} < \infty, 
\end{equation}
and thus
\[
 \frac{1}{\vp} = \lim_{n\ra\infty} \frac{T_n}{n} 
= \lim_{n\ra\infty} \frac{1}{n}\left(n + 2 \sum_{x<n} \mathcal{D}_n^x \right)
= 1 + \lim_{n\ra\infty} \frac{2}{n} \sum_{x=0}^{n-1} \mathcal{D}_n^x. 
\]
It follows from Lemma \ref{lem:bbpDxn} that $n^{-1} \sum_{x=0}^{n-1} \mathcal{D}_n^x$ has the same distribution as $n^{-1} \sum_{i=1}^n V_i$ started with $V_0 = 0$, and standard Markov chain arguments imply that
\[
 \lim_{n\ra\infty} \frac{1}{n} \sum_{i=1}^n V_i = \frac{ E_\eta^{V,0}\left[ \sum_{i=0}^{\s_0^V-1} V_i \right] }{ E_\eta^{V,0}[ \s_0^V ] }, \quad P_\eta^{V,0}\text{-a.s.} 
\]
Thus, we can conclude that 
\[
 \vp = \left( 1 + 2 \frac{ E_\eta^{V,0}\left[ \sum_{i=0}^{\s_0^V-1} V_i \right] }{ E_\eta^{V,0}[ \s_0^V ] } \right)^{-1} .
\]
Theorem \ref{th:bbp} implies that the $E_\eta^{V,0}[\s_0^V] < \infty$ (recall that $\d>1$ since the walk is transient to the right) and that $E_\eta^{V,0}[ \sum_{i=0}^{\s_0^V-1} V_i ] < \infty$ if and only if $\d>2$. 
Thus we can conclude that $\vp>0 \iff \d>2$. 
Again, a symmetric argument for random walks that are transient to the left shows that $\vp<0 \iff \tilde{\d} > 2$. 
\end{proof}

\begin{proof}[Proof of Theorem \ref{th:limdist}]
  The proofs of the limiting distributions in Theorem \ref{th:limdist}
  relies on the connection of the hitting times $T_n$ with the
  backward BLP $V_i$ from Lemma \ref{lem:bbpDxn} and the tail
  asymptotics for the backward BLP in \eqref{bbptails}.  We will give
  a brief sketch of the general argument here and will give full
  details in the case $\d=2$ in Appendix \ref{sec:delta2}.  We will
  refer the reader to previous papers for the details in all other
  cases.

For the limiting distributions of the hitting times $T_n$, recall that 
\begin{equation}\label{TnDxn}
 T_n = n + 2\sum_{x=0}^{n-1} \mathcal{D}_n^x + 2\sum_{x<0} \mathcal{D}_n^x. 
\end{equation}
It follows from \eqref{Dxnneg} that the third term on the right has a finite limit as $n\ra\infty$, $P_\eta$-a.s., 
and therefore it is enough to prove a limiting distribution for the first two terms on the right side of \eqref{TnDxn}. 
By Lemma \ref{lem:bbpDxn} this is equivalent to proving a limiting distribution for $n + 2\sum_{i=1}^n V_i$ under the measure $P_\eta^{V,0}$. 
The proof of the limiting distribution for the partial sums of the BLP relies on the regeneration structure of the process $V_i$. 
Let $r_k$ denote the time of the $k$-th return of the backward BLP to zero. That is, 
\[
 r_0 = 0 \quad \text{and} \quad r_k = \inf \{i>r_{k-1}: \, V_i = 0 \}.  
\]
Also, let $W_k = \sum_{i=r_{k-1}}^{r_k-1} V_i$. Note that $\{(r_k-r_{k-1}, W_k)\}_{k\geq 1}$ is an i.i.d.\ sequence under the measure $P_\eta^{V,0}$ and that Theorem \ref{th:bbp} implies that $r_1$ and $W_1$ are in the domains of attraction of totally asymmetric stable distributions of index $\min\{ \d, 2 \}$ and $ \min\{ \d/2, 2 \}$, respectively. 
From this, the limiting distributions for $n+2\sum_{i=1}^n V_i$ are standard.
For the details of the arguments in cases \ref{ldcase1}-\ref{ldcase5} we give the following references. 
\begin{itemize}
\item $\d \in (1,2)$: See \cite{bsRGCRW}, pages 847--849.
\item $\d = 2$: See Appendix \ref{sec:delta2}.
\item $\d \in (2,4]$: See Section 9 in \cite{kmLLCRW}.
\item $\d>4$: By (\ref{bbptails}) the second moment of the random variable
$\sum_{j=0}^{\sigma_0^V-1}V_j$ is finite. Thus, this is a classical
case covered by the standard CLT for the Markov chain $V$; see, for
example, \cite[I.16, Theorem 1]{cklMC}.\footnote{For the second
  statement of (v) see also \cite{kzPNERW}, proof of Theorem 3 and
  Section 6. The argument is based on the result due to A.-S.\
  Sznitman \cite[Theorem 4.1]{szSDRWRE} and gives the functional CLT
  for the position of the walk.}
\end{itemize}

Finally, to obtain the limiting distributions for the position $X_n$ of the ERW from the limiting distributions of the hitting times $T_n$, we use the fact that 
\begin{equation}\label{TXT}
 \{ T_m > n \} \subset \{ X_n < m \} \subset \{ T_{m+r} > n \} \cup \left\{ \inf_{k \geq T_{m+r}} X_k < m \right\},  
\end{equation}
for any $m,n,r\geq 1$. 
The key then is to control the probability of the last event on the right. To this end, it was shown in \cite[Lemma 6.1]{pERWLDP} that the tail asymptotics for $r_1$ in \eqref{bbptails} imply that 
\begin{equation}\label{backtrack}
 P_\eta\left( \inf_{k \geq T_{m+r}} X_k < m \right) \leq C r^{1-\d}, \quad \forall m,r \geq 1. 
\end{equation}
Again, for the details of how to use \eqref{TXT} and \eqref{backtrack} to obtain the limiting distributions for $X_n$, see the references given above. 
\end{proof}

\section{Mean and variance of the forward BLP}\label{sec:meanvar}

Many of the calculations below are simplified using matrix notation. 
\begin{itemize}
 \item $\one$ will denote a column vector of all ones, and $\mathbf{p}$ denotes a column vector with $i$-th entry $p(i)$. 
 \item $I$ is the identity matrix. 
 \item $D_p$ will denote a diagonal matrix with $i$-th diagonal entry $p(i)$. Similarly $D_{1-p} = I - D_p$ is the diagonal matrix with $i$-th diagonal entry $1-p(i)$. 
 \item Recall that $\bar{p}=\mu\cdot \mathbf{p}$, where $\mu$ is the stationary distribution for the environment Markov chain.
\end{itemize}
In this section and the next section, we will be concerned with a single increment of the BLP. Thus, we will only need to consider the environment at any fixed site $x\in \Z$. Therefore, in this section we will fix $x$ and suppress the sub/super-script $x$ for a less cumbersome notation. For instance, we will write $R_j = R_j^x$ for the Markov chain which generates the environment at $x$ and the Bernoulli sequence is denoted $\{\xi(j)\}_{j\geq 1} = \{\xi_x(j)\}_{j\geq 1}$. 

\subsection{Mean}
\begin{prop}\label{pr:Umean}
For every distribution $\eta$ on ${\cal R}$
 \[
 \lim_{n\ra\infty} \frac{E_\eta^{U,n}[U_1]}{n} =
 \frac{\bar{p}}{1-\bar{p}}=:\lambda.
 \]
 Moreover, there exist constants
 $\Cl[c]{ml}, \Cl[c]{mu} > 0$ such that for any $n\in\N$ and any distribution $\eta$
 on $\mathcal{R}$
\begin{equation}\label{meanlimit}
 \left| E_{\eta}^{U,n}[U_1] - \lambda n - \eta \cdot \mathbf{r} \right| \leq \Cr{ml} e^{-\Cr{mu} n},
\end{equation}
where 
\begin{equation}
  \label{eq:rform}
  \mathbf{r}=\mathbf{r}(\mathbf{p},K)=\left(I-K+\frac{(\mathbf{1-p})\mu D_{1-p}K}{1-\bar{p}}\right)^{-1}\mathbf{p}-\lambda\mathbf{1}.
\end{equation}
\end{prop}

\begin{proof}
Recall that $S_m$ is the number of successes before the $m$-th failure in the Bernoulli sequence $\xi = \{\xi(j)\}_{j\geq 1}$, and let $G_m = S_m - S_{m-1}$ be the number of successes between the $(m-1)$-st and the $m$-th failure in the Bernoulli sequence $\xi$. 
With this notation it follows from the construction of the forward BLP in Section \ref{sec:branching} that 
\begin{equation}\label{Umean}
 E_\eta^{U,n}[U_1 ] = \sum_{m=1}^n E_\eta[G_m].
\end{equation}
To compute $E_\eta[G_m]$ it helps to keep track of some additional information. 
Let $I_0 = R_1$ and for any $m\geq 1$ let $I_m$ be defined by
\[
 I_m =  R_{S_m +m+1}. 
\]
Note that $S_m + m$ is the number of Bernoulli trials needed to obtain $m$ failures. Therefore, $I_m = i$ if the next 
Bernoulli random variable after the $m$-th failure has success probability $p(i)$.  
Since 
\[
P_\eta( \{\xi(j)\}_{j \geq S_m + m + 1} \in \cdot \, | \, \s( \xi(j), \, j\leq S_m + m + 1) ) = P_{I_m}( \{\xi(j)\}_{j\geq 1} \in \cdot), 
\]
it follows that $\{I_m\}_{m\geq 0}$ is a Markov chain on the state
space $\mathcal{R}$. By the ellipticity assumption,
$p:{\cal R}\to (0,1)$, the Markov chain $\{I_m\}_{m\geq 0}$ has the
same unique closed irreducible subset $\mathcal{R}_0$ as the Markov
chain $\{R_j\}_{j\geq 1}$. Therefore, $\{I_m\}_{m\geq 0}$ has a unique
stationary distribution $\pi$. While we did not assume any
aperiodicity for $\{R_j\}_{j\ge 0}$, the ellipticity assumption
implies that $\{I_m\}_{m\geq 0}$ is aperiodic, and since ${\cal R}$ is
finite the convergence to stationarity is exponentially fast: 
if $\Pi$ is the matrix of transition probabilities for the Markov chain $I_m$, then
there
exist constants $\Cl[c]{sl},\Cl[c]{su}>0$ such that
\begin{equation}\label{pierror}
  \sup_{\eta} \| \eta \Pi^n - \pi \|_\infty \leq \Cr{sl} e^{-\Cr{su} n}, \quad \forall n\geq 1,
\end{equation}
where the supremum on the left is over all probability measures $\eta$ on $\mathcal{R}$.

Let $\mathbf{g}$ denote the column vector of length $N$ with $i$-th
entry $g(i) = E_i[G_1]$.  Note that by the comparison with 
a geometric random variable with parameter
$p_{\text{max}}:=\max_{i\in{\cal R}} p(i)<1$ we see that
$g(i)<p_{\text{max}}/(1-p_{\text{max}})<\infty$ for every $i \in \mathcal{R}$.  Then, since
$E_{\eta}[ G_m ] = E_{\eta}[ g(I_{m-1}) ]$ it follows from an ergodic
theorem for finite state Markov chains that
\[
 \lim_{n\ra\infty} \frac{1}{n} E_\eta^{U,n}[U_1]=\lim_{n\ra\infty} \frac{1}{n} \sum_{m=1}^{n}E_{\eta}[ G_m ] = E_\pi[ g(I_0) ] = \sum_{i \in \mathcal{R}} \pi(i) g(i) = \pi \cdot \mathbf{g}.  
\]
Therefore, to prove the first part of the lemma we need to show that
$\pi\cdot\mathbf{g}=\lambda$.
The following lemma accomplishes this task. It also provides useful
information about $\{I_m\}_{m\ge 0}$ which we need in the rest of this section.
\begin{lem}\label{lem:Pipi}
The sequence $\{I_m\}_{m\geq 0}$ is a Markov chain 
with transition probabilities given by the matrix
\[
 \Pi 
= (I-D_p K)^{-1}D_{1-p}K, 
\]
and with a unique stationary distribution $\pi$ given by 
\begin{equation}\label{piform}
  \pi 
= \frac{\mu(I-D_p K)}{1-\bar{p}}
= (1+\l) \mu(I-D_p K). 
\end{equation}
Moreover,
\begin{equation}
  \label{eq:gform}
  \mathbf{g}
= (I-D_p K)^{-1} \mathbf{p}\quad\text{and}\quad \pi\cdot\mathbf{g}=\lambda. 
\end{equation}
\end{lem}
\begin{rem}

In the first formula for $\pi$ in \eqref{piform}, the multiplicative factor $1/(1-\bar{p})$ is needed for the entries of $\pi$ to sum to 1 since 
 $\mu(I-D_p K) \one = \mu(\one - \mathbf{p})  = 1-\bar{p}$. 
The second formula for $\pi$ in \eqref{piform} is equivalent since $\l = \frac{\bar{p}}{1-\bar{p}}$ implies that $\frac{1}{1-\bar{p}} = 1+\l$. 
\end{rem}

\begin{proof}
To compute the transition probabilities, for $k\geq 0$ let $M_k$ be the $N\times N$ matrix with entries
\[
 M_k(i,i') = P_i(G_1 = k, I_1 = i'). 
\]
Obviously, $M_0(i,i') = (1-p_i)K_{i,i'} = (D_{1-p}K)_{i,i'}$ and for $k\geq 1$ by conditioning on the value of $R_2$ we obtain that
\[
 M_k(i,i') = \sum_{\ell=1}^N p_i K_{i,\ell} M_{k-1}(\ell,i'). 
\]
That is, $M_0 = D_{1-p} K$ and $M_k = D_p K M_{k-1}$ for $k\geq 1$. 
Combining these we get that $M_k = (D_p K)^k D_{1-p} K$. 
Since $\Pi(i,i') = \sum_{k\geq 0} M_k(i,i')$, we have that 
\[
\Pi = \sum_{k=0}^\infty M_k=\sum_{k=0}^\infty (D_p K)^k D_{1-p}K =
(I-D_p K)^{-1}D_{1-p}K.
\]
(Note that since $D_p K $ is a matrix with non-negative entries and with $i$-th row sum equal to $p(i)<1$, the Perron-Frobenius theorem implies that all eigenvalues of $D_p K$ have absolute value strictly less than one, and thus $I-D_p K$ is invertible.)
It it is easy to check that $\pi \Pi = \pi$ by noting that $\mu K = \mu$ and, thus,
\begin{equation}\label{muDpK}
 \mu D_{1-p} K = \mu (I-D_p)K = \mu(I-D_p K).
\end{equation}

Finally, we give a formula for $\mathbf{g}$. For $k\geq 1$ one easily
sees that
\[P_i(G_1 \geq k) = \sum_{i_1 \in \mathcal{R}} p(i) K_{i,i_1}
P_{i_1}(G_1 \geq k-1).\] Iterating this, we obtain
\begin{equation}\label{Gtail}
 P_i(G_1 \geq k) = \sum_{i_1,i_2,\ldots,i_k \in \mathcal{R}} (p(i) K_{i,i_1})(p(i_1) K_{i_1,i_2})\cdots(p(i_{k-1}) K_{i_{k-1},i_k} ) = e_i (D_p K)^k \mathbf{1}, 
\end{equation}
where in the last equality we use the notation $e_i$ for the row vector with a one in the $i$-th coordinate and zeros elsewhere. Therefore,
\[
 g(i) = E_i[G_1] = \sum_{k\geq 1} P_i(G_1 \geq k)= \sum_{k\geq 1} e_i (D_p K)^k \mathbf{1} = e_i (I-D_p K)^{-1} D_p K \one = e_i (I-D_p K)^{-1} \mathbf{p} 
\]
and we get (\ref{eq:gform}) as claimed. From this and the formula for $\pi$ in \eqref{piform} it follows immediately that $\pi \cdot \mathbf{g} = \l$. 
\end{proof}
In the critical case $\bar{p} = 1/2$, Lemma \ref{lem:Pipi} and \eqref{muDpK} give the following simpler formula for the stationary distribution $\pi$ that will be useful below. 
\begin{cor}\label{picrit}
  If $\bar{p} = 1/2$ then 
  $\pi = 2 \mu(I-D_p K) = 2\mu D_{1-p} K$.
\end{cor}

Thus far we have proved the first part of Proposition~\ref{pr:Umean}. Next we show the existence of a vector $\mathbf{r}$ such that \eqref{meanlimit} holds.  
To this end, for any $n\geq 1$ let $\mathbf{r}_n = (r_n(1),r_n(2),\ldots,r_n(N))^t$ be the column vector with $i$-th entry \[r_n(i)=  E_i^{U,n}[ U_1] - \lambda n
= \sum_{m=1}^n \left( E_i[G_m] - \lambda \right).\] Then
\begin{equation*}
  \mathbf{r}_n = \sum_{m=1}^n \left( \Pi^{m-1} \mathbf{g} - \one \pi\cdot \mathbf{g} \right) 
= \sum_{m=0}^{n-1} \left( \Pi^{m} - \one \pi \right) \mathbf{g},
\end{equation*}
where in the last equality the matrix $\one \pi$ is the $N\times N$
matrix with all rows equal to the vector $\pi$ which is the stationary
distribution for $\{I_m\}_{m\ge 0}$.  It follows from \eqref{pierror} that the
entries of $\Pi^m - \one \pi$ decrease exponentially in $m$, so that
the sum in the last line converges as $n\ra\infty$. That is,
\begin{equation}\label{rnlim}
 \lim_{n\ra\infty} \mathbf{r}_n = \mathbf{r} := \sum_{m=0}^{\infty} \left( \Pi^{m} - \one \pi \right) \mathbf{g}. 
\end{equation}
Moreover, since
$E_{\eta}^{U,n}[ U_1 ] - \lambda n = \eta \cdot
\mathbf{r}_n$, then
\[
  \left|E_{\eta}^{U,n}[ U_1 ] - \lambda n - \eta \cdot \mathbf{r} \right| 
= \left| \eta \cdot \left( \mathbf{r}_n - \mathbf{r} \right) \right| 
\leq \sum_{m=n}^\infty \left\| \eta \Pi^m - \pi \right\|_\infty \|\mathbf{g}\|_1 \leq 
\frac{\|\mathbf{g}\|_1 \Cr{sl}}{1-e^{-\Cr{su}}} e^{-\Cr{su} n},
\]
where the last inequality follows from \eqref{pierror}. 

Finally, we give an explicit formula for the vector $\mathbf{r}$. To this end, since $\pi \Pi = \pi$ it follows that $(\one \pi) \Pi = \one \pi  = \Pi (\one \pi)$ and thus $(\Pi^m - \one \pi) = (\Pi - \one \pi)^m$ for all $m\geq 1$. Therefore, 
\begin{equation}\label{eq:rform0}
  \mathbf{r} = (I-\mathbf{1}\pi)\mathbf{g} + \sum_{m=1}^{\infty} \left( \Pi - \one \pi \right)^m \mathbf{g}
  = \sum_{m=0}^{\infty} \left( \Pi - \one \pi \right)^m \mathbf{g}-(\pi\cdot\mathbf{g})\mathbf{1}= (I - \Pi + \one \pi )^{-1} \mathbf{g}-\lambda\mathbf{1}.
\end{equation}
Substituting the expressions for $\pi,\ \Pi$, and $\mathbf{g}$  and simplifying we obtain (\ref{eq:rform}).
\end{proof}
In closing this subsection, we note the following Corollary which will be of use later. 
\begin{cor}\label{pir}
  $\pi\cdot \mathbf{r}=0$.
\end{cor}
\begin{proof}

 Since $\pi$ is the stationary distribution for the Markov chain with transition matrix $\Pi$ it follows that $\pi(I-\Pi+\one \pi) = \pi$, or equivalently $\pi(I-\Pi+\one\pi)^{-1} = \pi$. Using the formula for $\mathbf{r}$ in \eqref{eq:rform0}, it follows that 
 $\pi \cdot \mathbf{r} = \pi(I-\Pi+\one\pi)^{-1}\mathbf{g} - \l \pi \cdot \one = \pi \cdot \mathbf{g} - \l = 0$, where the last equality follows from \eqref{eq:gform}. 

\end{proof}

\subsection{Variance}
The main result of this subsection is the following proposition. 
\begin{prop}\label{varerror}
  There exists a constant $\Cl[c]{v}>0$ such that for any distribution
  $\eta$ on ${\cal R}$
\[
 \left| \frac{\Var_\eta( U_1 | \, U_0 = n)}{n} - \nu \right| \leq \frac{\Cr{v}}{n}, 
\]
where the constant $\nu$ is given by the formula 
\begin{equation}\label{nudef}
  \nu=\nu(\mathbf{p},K)=\Var_{\pi}(G_1) + 2 \sum_{k=1}^{\infty} 
\Cov_{\pi}(G_1,G_{1+k}) 
=(1+\l)(\l  + 2 \mu  D_p K \mathbf{r}),
\end{equation}
and $\mathbf{r}$ is the vector from Proposition \ref{pr:Umean}. 
In particular, if $\bar{p}=1/2$ then 
\begin{equation}
  \label{nucrit}
  \nu = 2+4\mu D_pK \mathbf{r}=4\mu D_pK(I-K+2\mathbf{(1-p)}\mu D_{1-p}K)^{-1}\mathbf{p}.
\end{equation}
\end{prop} 



\begin{proof}
First note that for any measure $\eta$ on ${\cal R}$,  
\begin{align*}
 \Var_\eta( U_1 | \, U_0 = n) 
 = \Var_\eta\left( \sum_{k=1}^n G_k \right) 
 &= \sum_{k=1}^n \Var_\eta\left( G_k \right) + 2 \sum_{1\leq k < \ell \leq n} \Cov_\eta( G_k, G_\ell). 
\end{align*}
Let $\nu_n = \Var_{\pi}( U_1 | \, U_0 = n) = \Var_{\pi}\left( \sum_{k=1}^n G_k \right)$. The proof will consist of three steps.
\begin{itemize}
\item [\textit{Step 1.}] Show that $\left|\nu_n-\Var_\eta( U_1 | \, U_0 = n)\right|$ is bounded by a constant uniformly in $n$ and $\eta$.
\item [\textit{Step 2.}] Prove that there is a constant $C>0$ such that
  $\left|\nu_n/n-\nu\right|\le C/n$.
\item [\textit{Step 3.}] Calculate
  $\nu$ explicitly and show that (\ref{nudef}) and (\ref{nucrit}) hold.
\end{itemize}

\medskip

\textit{Step 1.} For any $k\geq 0$ and $i\in{\cal R}$ let
$v_k(i) = E_i[ G_1 G_{1+k} ]$ and
\[\mathbf{v}_k = (v_k(1),v_k(2),\ldots,v_k(N))^t.\]
With this notation, we have that
\begin{equation}\label{VarmuG}
 \Var_\eta(G_k) = \eta \Pi^{k-1} \mathbf{v}_0 - (\eta \Pi^{k-1} \mathbf{g} )^2, \quad \forall k\geq 1, 
\end{equation}
and 
\begin{equation}\label{CovmuG}
 \Cov_\eta(G_k,G_\ell) = \eta\Pi^{k-1}\mathbf{v}_{\ell-k} - (\eta \Pi^{k-1} \mathbf{g})(\eta \Pi^{\ell-1} \mathbf{g}), \quad \forall 1\leq k < \ell. 
\end{equation}
Note that in the special case where $\eta$ has distribution $\pi$ these become 
\begin{equation}\label{VarCovpi}
  \Var_{\pi}(G_k) = \pi (\mathbf{v}_0 - \lambda \mathbf{g}), \quad \text{and} \quad \Cov_{\pi}(G_k,G_\ell) = \pi (\mathbf{v}_{\ell-k} - \lambda \mathbf{g}), \quad \text{for } 1 \leq k < \ell, 
\end{equation}
since $\pi \Pi = \pi$ and $\pi\cdot \mathbf{g}=\lambda$. The following lemma is elementary.
\begin{lem}
$\sup_{k\geq 0} \|\mathbf{v}_k \|_\infty < \infty$. 
\end{lem}
\begin{proof}
 First of all, note that the Cauchy-Schwartz inequality implies that
\[
 E_i[G_1 G_{1+k}] \leq \sqrt{ E_i[G_1^2] E_i[G_{1+k}^2] }\leq \max_i E_i[G_1^2] = \|\mathbf{v}_0\|_\infty, 
\]
and thus it is enough to prove that $E_i[G_1^2] < \infty$ for every
$i \in \mathcal{R}$. The last inequality is obvious by comparison with a geometric random variable with parameter $p_{\text{max}}=\max_{i\in{\cal R}}p(i)<1$.
\end{proof}

\begin{lem}\label{lem:VarCovdiff}
 There exist constants, $\Cl[c]{vl}, \Cl[c]{vu} > 0$ so that for any distribution $\eta$ on ${\cal R}$ and any $1 \leq k < \ell$
\begin{equation}\label{VarCovdiff}
 \left| \Var_\eta(G_k) - \Var_{\pi}(G_k) \right| \leq \Cr{vl} e^{-\Cr{vu} k}
\quad\text{and}\quad
 \left| \Cov_\eta(G_k,G_\ell) - \Cov_{\pi}(G_k,G_\ell) \right| \leq \Cr{vl} e^{-\Cr{vu} \ell}.
\end{equation}
 \end{lem}
\begin{proof}
  The key observation is the exponential convergence to the stationary
  distribution $\pi$ for the Markov chain $\{I_m\}_{m\ge 0}$ as noted
  in \eqref{pierror} above.  From this, it follows easily that
\[
 | \eta \Pi^{k-1} \mathbf{v}_0 - \pi \cdot \mathbf{v}_0 | \leq \Cr{sl} \|\mathbf{v}_0\|_1 e^{-\Cr{su} (k-1)}
\quad\text{and}\quad
 | \eta \Pi^{k-1} \mathbf{g} - \pi \cdot \mathbf{g}| \leq \Cr{sl} \|\mathbf{g}\|_1 e^{-\Cr{su} (k-1)}.
\]
The first inequality in \eqref{VarCovdiff} then follows easily from
the above bounds and the representations for the variances in
\eqref{VarmuG} and \eqref{VarCovpi} taking into account the fact that
$\eta\Pi^{k-1}$ is always a probability distribution so that
$| \eta \Pi^{k-1} \mathbf{g} + \pi \cdot \mathbf{g}|\le
2\|\mathbf{g}\|_\infty$.

To obtain the bound on the difference of the covariance terms in \eqref{VarCovdiff}, note that the representations in \eqref{CovmuG} and \eqref{VarCovpi} imply that 
\begin{align*}
  \left| \Cov_\eta (G_k,G_\ell) - \Cov_{\pi}(G_k,G_\ell) \right| 
  &\leq \left| \eta\Pi^{k-1}\mathbf{v}_{\ell-k} - (\eta \Pi^{k-1} \mathbf{g})(\eta \Pi^{\ell-1} \mathbf{g}) - \pi (\mathbf{v}_{\ell-k} -\lambda\mathbf{g}) \right| \\
  &\leq \left| (\eta \Pi^{k-1}-\pi)( \mathbf{v}_{\ell-k} - \lambda\mathbf{g} )  \right| 
    + \left| \eta\Pi^{k-1} \mathbf{g}\right| \left| \lambda - \eta \Pi^{\ell-1} \mathbf{g}  \right| \\
  &\leq \Cr{sl} e^{-\Cr{su} (k-1)} \|\mathbf{v}_{\ell-k} - \lambda \mathbf{g} \|_1 + \|\mathbf{g} \|_\infty \|\mathbf{g}\|_1 \Cr{sl} e^{-\Cr{su} (\ell-1)},
\end{align*}
where the last inequality again follows from \eqref{pierror} and the fact that $\pi \cdot \mathbf{g} = \lambda$. 
Therefore, it will be enough to show that there exist constants $C, C'>0$ such that
\begin{equation}\label{vkgdiff}
 \|\mathbf{v}_k - \lambda\mathbf{g} \|_1 \leq  C \|\mathbf{g}\|_1^2e^{-C' k}, \quad \forall k\geq 1. 
\end{equation}
To this end, note that by conditioning on $(G_1,I_1)$ we get that 
\[
E_i[G_1 G_{k+1}] - \lambda E_i[G_1] = \sum_{i'\in\mathcal{R}} \sum_{n=0}^\infty n P_i(G_1 = n, \, I_1 = i') \left( E_{i'}[G_k] - \lambda\right).
\]
By (\ref{pierror})
$\max_{i \in \mathcal{R}} |E_{i}[G_k] - \lambda| = \max_{i \in \mathcal{R}} | e_{i} \Pi^{k-1} \mathbf{g} - \pi \cdot \mathbf{g}| \leq \Cr{sl}
\|\mathbf{g}\|_1 e^{-\Cr{su} (k-1)}$, and thus it follows that
\[
 \left| E_i[G_1 G_{k+1}] - \lambda E_i[G_1] \right| \leq E_i[G_1] \left( \Cr{sl} \|\mathbf{g}\|_1 e^{-\Cr{su} (k-1)} \right).
\]
Since this gives a bound on each of the entries of $\mathbf{v}_k - \lambda \mathbf{g}$, the inequality in \eqref{vkgdiff} follows. 
\end{proof}
To complete step 1 in the proof of Proposition \ref{varerror} we notice that Lemma \ref{lem:VarCovdiff} implies 
\[
 \left|\nu_n-\Var_\eta( U_1 | \, U_0 = n)\right|=\left|\Var_{\pi} \left( \sum_{k=1}^n G_k\right) - \Var_\eta\left( \sum_{k=1}^n G_k \right)   \right| 
\leq \sum_{k=1}^n \Cr{vl} e^{-\Cr{vu} k} + 2 \sum_{1\leq k < \ell \leq n} \Cr{vl} e^{-\Cr{vu} \ell}.
\]
Note that the sums on the right are uniformly bounded in $n$ and that the constants $\Cr{vl}, \Cr{vu}$ do not depend on $\eta$. Thus, we conclude that 
\begin{equation}\label{mupivar}
 \sup_{n, \eta} \left| \Var_\eta\left( \sum_{k=1}^n G_k \right) - \Var_{\pi}\left( \sum_{k=1}^n G_k \right)\right| < \infty. 
\end{equation}

\textit{Step 2.} Note that 
\begin{equation*}
 \frac{\nu_n}{n} 
 = \frac{1}{n} \sum_{k=1}^n \Var_{\pi}(G_k) + \frac{2}{n} \sum_{1\leq k<\ell \leq n} \Cov_{\pi}(G_k,G_\ell) 
 = \Var_{\pi}(G_1) + 2 \sum_{k=1}^{n-1} \left(1-\frac{k}{n}\right)\Cov_{\pi}(G_1,G_{1+k}),
\end{equation*}
where in the last equality the change in the indices is due to the fact that $\pi$ is the stationary distribution for the Markov chain $\{I_m\}_{m\ge 0}$. 
Therefore, 
\begin{align*}
 \Var_{\pi}(G_1) +2 \sum_{k=1}^{\infty} \Cov_{\pi}(G_1,G_{1+k})- \frac{\nu_n}{n} 
 &= 2 \sum_{k=n}^{\infty} \Cov_{\pi}(G_1,G_{1+k}) + 2 \sum_{k=1}^{n-1} \frac{k}{n} \Cov_{\pi}(G_1,G_{1+k}). 
\end{align*}
It follows from \eqref{VarCovpi} and \eqref{vkgdiff} that 
\[
 \left| \Cov_{\pi}(G_1,G_{1+k}) \right| 
= \left| \pi \cdot( \mathbf{v}_k - \lambda \mathbf{g} ) \right| 
\leq C e^{-C' k}. 
\]
Therefore, for some $C''>0$
\[
 \left|\Var_{\pi}(G_1) +2 \sum_{k=1}^{\infty} \Cov_{\pi}(G_1,G_{1+k})- \frac{\nu_n}{n} \right| \leq 2 \sum_{k=n}^{\infty} C e^{-C' k} + \frac{2}{n} \sum_{k=1}^{n-1} C k e^{-C' k} \leq \frac{C''}{n}.
\]

\textit{Step 3.} 
Since $E_i[G_1^2] = \sum_{k=1}^\infty (2k-1) P_i(G_1 \geq k)$, recalling the  formula in \eqref{Gtail} for $P_i(G_1 \geq k)$ we obtain 
\begin{align}
 \mathbf{v}_0 &= \sum_{k=1}^\infty (2k-1)(D_p K)^k \one =  (I + D_p K) (I-D_p K)^{-2} D_p K \one \nonumber\\
&= (I-D_p K)^{-1} (I + D_p K)(I-D_p K)^{-1} \mathbf{p} = (I-D_p K)^{-1} (I + D_p K) \mathbf{g}. \label{v0form}
\end{align}
Using \eqref{v0form} and \eqref{VarCovpi}, we have that 
\begin{align*}
 \Var_{\pi}(G_1)
= \pi\left(  (I-D_p K)^{-1} (I+D_p K) - \l I \right) \mathbf{g}
= \pi &(I-D_p K)^{-1} ((1-\l)I + (1+\l)D_p K) \mathbf{g} \\
= (1-\l) \pi &(I-D_p K)^{-1} \mathbf{g} + (1+\l) \pi D_p K (I-D_p K)^{-1} \mathbf{g}.  
\end{align*}
Since $\pi = (1+\l)\mu (I-D_p K)$ this simplifies to 
\begin{equation}\label{VarpiG}
 \Var_{\pi}(G_1)
=(1-\l^2) \mu \mathbf{g} + (1+\l)^2 \mu D_p K \mathbf{g}. 
\end{equation}

To compute $\Cov_{\pi}(G_1,G_{1+k})$ we  need a formula for $\mathbf{v}_k$. 
Recall that $M_n(i,i') = P_i(G_1=n, I_1=i')$, so that by conditioning on $G_1$ and $I_1$ we obtain 
\begin{align*}
 E_i[G_1 G_{1+k}] &= \sum_{n=0}^\infty n M_n(i,i')E_{i'}[G_k] = \sum_{n=0}^\infty n e_i \left( M_n \Pi^{k-1} \mathbf{g} \right).
\end{align*}
Using this and the fact that $M_n = (D_p K)^n D_{1-p}K$ we get
\begin{align*}
 \mathbf{v}_k = \sum_{n=0}^\infty n \left( M_n \Pi^{k-1} \mathbf{g} \right) 
&= \left(\sum_{n=0}^\infty n (D_p K)^n \right) D_{1-p}K \Pi^{k-1} \mathbf{g}  \\
&= D_p K (I-D_p K)^{-2} D_{1-p} K \Pi^{k-1} \mathbf{g}.
\end{align*}
Now, note that if in this equation the matrix $\Pi^{k-1}$ is replaced by $ \one \pi$ (recall this is the matrix with all rows equal to $\pi$) then 
since $\pi \mathbf{g} = \l$ we have
\begin{align*}
  D_p K (I-D_p K)^{-2} D_{1-p} K \one \pi \mathbf{g}
&= \l D_p K (I-D_p K)^{-2}  D_{1-p}K \one 
= \l D_p K (I-D_p K)^{-2} (\one - \mathbf{p}) \\
&= \l (I-D_p K)^{-1} D_p K (I-D_p K)^{-1} (\one - \mathbf{p}) 
= \l (I-D_p K)^{-1} D_p K \one \\
&= \l (I-D_p K)^{-1} \mathbf{p} 
= \l \mathbf{g}. 
\end{align*}
Therefore, we can re-write the formula for $\Cov_{\pi}(G_1,G_{1+k})$ in \eqref{VarCovpi} as 
\[
 \Cov_{\pi}(G_1,G_{1+k}) = \pi \cdot( \mathbf{v}_k - \l \mathbf{g} )
= \pi D_p K(I-D_p K)^{-2}  D_{1-p} K\left( \Pi^{k-1}  - \one \pi \right) \mathbf{g}. 
\]
Recalling the definition of $\mathbf{r}$ in \eqref{rnlim}, this implies that 
\begin{align}
 \sum_{k=1}^\infty \Cov_{\pi}(G_1,G_{1+k}) &= \pi D_p K(I-D_p K)^{-2}  D_{1-p} K \mathbf{r} \nonumber \\
&= \pi (I-D_p K)^{-1}D_p K(I-D_p K)^{-1}  D_{1-p} K \mathbf{r} 
= (1+\l) \mu D_p K \Pi \mathbf{r}, \label{Covsum}
\end{align}
where the last equality follows from Lemma \ref{lem:Pipi}. 

Combining \eqref{VarpiG} and \eqref{Covsum}, we obtain that 
\begin{align*}
 \nu = (1-\l^2) \mu \mathbf{g} + (1+\l)^2 \mu D_p K \mathbf{g} + 2 (1+\l) \mu  D_p K \Pi \mathbf{r}.
\end{align*}
To further simplify this, note that it follows from \eqref{eq:rform0} and Corollary \ref{pir} that
\[
 \mathbf{g} = (I-\Pi+\one \pi)(\mathbf{r} + \l \one) = (I-\Pi) \mathbf{r} + \l \one.  
\]
Re-arranging this we get $\Pi \mathbf{r} = \mathbf{r} + \l \one - \mathbf{g}$, and putting this back into the above formula for $\nu$ we get 
\begin{align*}
 \nu &= (1-\l^2) \mu \mathbf{g} + (1+\l)^2 \mu D_p K \mathbf{g} + 2 (1+\l) \mu  D_p K (  \mathbf{r} + \l \one - \mathbf{g} ) \\
&= (1-\l^2) \mu \mathbf{g} + (\l^2-1) \mu D_p K \mathbf{g} + 2 (1+\l) \mu  D_p K \mathbf{r} + 2 (1+\l)\l \mu  D_p K \one \\
&= (1-\l^2) \mu (I-D_p K) \mathbf{g} + 2 (1+\l) \mu  D_p K \mathbf{r} + 2 (1+\l)\l \bar{p}  \\
&= (1-\l^2) \mu\cdot \mathbf{p} + 2 (1+\l) \mu  D_p K \mathbf{r} + 2 (1+\l)\l \bar{p} 
= (1+\l)\l  + 2 (1+\l) \mu  D_p K \mathbf{r} 
\end{align*}
where in the second to last equality we used the formula \eqref{eq:gform} for $\mathbf{g}$, and in the last equality we used that $\mu \cdot \mathbf{p} = \bar{p} = \frac{\l}{1+\l}$. 
\end{proof}

\section{The backward BLP and parameter relationships}
\label{sec:param}

Throughout this section we will assume that we are in the critical
case $\bar{p} = 1/2$. We shall discuss the backward BLP, give explicit
formulas for parameters $\delta$ and $\tilde{\delta}$, and derive the
key relationship between them.

Consider first the backward BLP $V$.  To obtain the results about $V$ from those for the forward BLP $U$
\begin{itemize}
\item we need to replace $\bp$ with
$\one-\bp$ everywhere in order to switch from counting
``successes'' to counting ``failures'';
\item we have to account for the fact that $V$ has one ``immigrant''
  in each generation: recall that $V_i=F^i_{V_{i-1}+1}$ while
  $U_i=S^i_{U_{i-1}}$, $i\in\N$.
\end{itemize}

The above observations lead to the following
statements whose proofs are identical to those for $U$. We shall
state the results only for the critical case, since we use $V$ solely in
the critical setting.
\begin{prop}\label{pr:Vmean}
Let $\bar{p}=1/2$. For every distribution $\eta$ on ${\cal R}$
 \[
 \lim_{n\ra\infty} \frac{E_\eta[V_1| \, V_0 = n]}{n} = 1.
 \]
 Moreover, there exist constants
 $\Cl[c]{nl}, \Cl[c]{nu} > 0$ such that for any $n\in\N$ and any distribution $\eta$
 on $\mathcal{R}$
\begin{equation*}
 \left| E_{\eta}[V_1| \, V_0 = n] - n - (1+\eta \cdot \tilde{\mathbf{r}}) \right| \leq \Cr{nl} e^{-\Cr{nu} n},
\end{equation*}
where 
\begin{equation*}
  \tilde{\mathbf{r}}=\mathbf{r}(\mathbf{1-p},K)=\left(I-K+2\mathbf{p}\mu D_pK\right)^{-1}(\mathbf{1-p})-\mathbf{1}.
\end{equation*}
\end{prop}
\begin{prop}\label{Vvarerror}
  Let $\bar{p}=1/2$. There exists a constant $\Cl[c]{w}>0$ such that for any
  distribution $\eta$ on ${\cal R}$
\[
 \left| \frac{\Var_\eta( V_1 | \, V_0 = n)}{n} - \tilde{\nu} \right| \leq \frac{\Cr{w}}{n}, 
\]
where the constant $\tilde{\nu}$ is given by the formula 
\begin{equation}
  \label{tnucrit}
  \tilde{\nu} = \nu(\mathbf{1-p},K)=2+4\mu D_{1-p}K \tilde{\mathbf{r}}=4\mu D_{1-p}K(I-K+2\mathbf{p}\mu D_pK)^{-1}(\mathbf{1-p}).
\end{equation}
\end{prop} 
 
For reader's convenience we list the relevant parameters for $U$ and $V$ side by side.
\begin{alignat}{2}
  \pi&=2\mu(1-D_p K)=2\mu D_{1-p}K; \quad &&\tilde{\pi}=2\mu(1-D_{1-p}K)=2\mu D_pK; \nonumber \\
  \mathbf{r}&=(I-K+2\mathbf{(1-p)} \mu D_{1-p}K)^{-1}\mathbf{p}-\one;
  \quad\quad
  &&\tilde{\mathbf{r}}=(I-K+2\mathbf{p} \mu D_pK)^{-1}\mathbf{(1-p)}-\one; \label{parforms}\\
  \nu&=2+4\mu D_pK \mathbf{r}=2+2\tilde{\pi}\br; \quad
  &&\tilde{\nu}=2+4\mu D_{1-p}K \tilde{\mathbf{r}}=2+2\pi\tilde{\br}. \nonumber
\end{alignat}
In the above formulas, recall that 
\begin{itemize}
 \item $K$ is the transition matrix for the Markov chain $R_j$ used to generate the cookie environment at each site. The row vector $\mu$ is the stationary distribution for this Markov chain. 
 \item $\mathbf{p}$ is the column vector with $i$-th entry $p(i)$, $\one$ is a  column vector of all ones, and $D_p$ and $D_{1-p}$ are diagonal matrices with $i$-th entry $p(i)$ and $1-p(i)$, respectively.
 \item The row vectors $\pi$ and $\tilde\pi$ give the limiting distribution of the next ``cookie'' to be used after the $n$-th failure or success, respectively, in the sequence of Bernoulli trials $\xi(j)$ at a site. 
 \item The column vector $\mathbf{r}$ and the parameter $\nu$ are related to the asymptotics of the mean and variance of the forward BLP $U$ as given in Propositions \ref{pr:Umean} and \ref{varerror}, and $\tilde{\mathbf{r}}$ and $\tilde\nu$ are related to the asymptotics of the backward BLP $V$ in the same way by Propositions \ref{pr:Vmean} and \ref{Vvarerror}. 
\end{itemize}

\noindent
Now we can define the parameters $\d$ and $\tilde\d$ that appear in Theorems \ref{th:rectran}--\ref{th:limdist}.
\begin{equation}
  \label{eq:del}
  \delta=\frac{2\eta\cdot \br}{\nu},\quad 
\tilde{\delta}=\frac{2\eta\cdot\tilde{\br}}{\tilde{\nu}}.
\end{equation}
Note that the parameter $\d$ can be computed in terms of $\mathbf{p}$, $K$ and $\eta$. If we wish to make this dependence explicit we will write $\d = \d(\mathbf{p},K,\eta)$. In particular, with this notation we have that 
$\tilde{\d} = \d(\one - \mathbf{p},K,\eta)$.

We close this section with a justification of the statement of Theorem \ref{th:rectran} that $\d + \tilde\d<1$. 
\begin{prop}
  \label{deleq}
$\nu=\tilde\nu$ and $\delta+\tilde{\delta}=1-2/\nu$.
\end{prop}

\begin{proof}
We shall need the following lemma.
\begin{lem}\label{lem:rtrsum}
  $\br+\tilde{\br}=(\mu(\br+\tilde{\br}))\one$. In particular, $\br + \tilde{\br}$ is a constant multiple of $\mathbf{1}$. 
\end{lem}
Let us postpone the proof of Lemma~\ref{lem:rtrsum} and continue with
the proof of Proposition~\ref{deleq}.
Recall that by Corollary~\ref{pir} $\pi\cdot\br=0$. Similarly,
$\tilde{\pi}\cdot \tilde{\br}=0$.  Therefore,
Lemma~\ref{lem:rtrsum} and the formulas for $\nu$ and $\tilde\nu$ in \eqref{parforms} imply that 
\[
 \tilde{\nu}
=2+2\pi(\br+\tilde{\br})
\overset{\mathrm{L.\,}\ref{lem:rtrsum}}{=}2+2\mu(\br+\tilde{\br})
\overset{\mathrm{L.\,}\ref{lem:rtrsum}}{=}2+2\tilde{\pi}(\br+\tilde{\br})
=\nu. 
\]
Moreover, from the above line we see that
$\mu(\br+\tilde{\br})=\nu/2-1$. Combining this with
Lemma~\ref{lem:rtrsum} we get that
$\br+\tilde{\br}=(\nu/2-1)\one$. Then by \eqref{eq:del}  we have $\delta+\tilde{\delta}=2\eta\cdot(\mathbf{r+\tilde{r}})/\nu=1-2/\nu$ as claimed.
\end{proof}

\begin{proof}[Proof of Lemma~\ref{lem:rtrsum}]
Recall that $\mathbf{r}_n$ is the column vector with $i$-th entry 
\[
 r_n(i) = E_{i}^{U,n}[ U_1 ] - n
= E_i[ S_n] - n,
\]
where $S_n$ is the number of successes before the $n$-th failure in the sequence of Bernoulli trials $\{\xi(j)\}_{j\geq 1}$. 
Let $Z_n = \sum_{j=1}^n \xi(j)$ be the number of successes in the first $n$ Bernoulli trials. 
If $Z_n = k$, then there are $n-k$ failures among the first $n$ Bernoulli trials and so $S_n$ is equal to $k$ plus the number of successes before the $k$-th failure in shifted Bernoulli sequence $\{\xi(j)\}_{j\geq n+1}$. 
Therefore, by conditioning on $Z_n$ and $R_{n+1}$ (the type of the $(n+1)$-st cookie), we obtain that 
\begin{align}
 r_n(i) 
&= \sum_{i'\in \mathcal{R}} \sum_{k=0}^n P_i( Z_n = k, \, R_{n+1} = i' )\left( k + E_{i'}\left[ S_k \right] - n \right) \nonumber \\
&= \sum_{i'\in \mathcal{R}} \sum_{k=0}^n P_i( Z_n = k, \, R_{n+1} = i' )\left( 2k + r_k(i') - n \right) \nonumber \\
&= E_i\left[ 2Z_n - n \right] + \sum_{i'\in \mathcal{R}} \sum_{k=0}^n P_i( Z_n = k, \, R_{n+1} = i' )r_k(i') \nonumber \\
& =  \sum_{j=1}^n  \E_i\left[2p(R_j) - 1 \right]  +  e_i K^n \mathbf{r} + \sum_{i'\in \mathcal{R}} \sum_{k=0}^n P_i( Z_n = k, \, R_{n+1} = i' )( r_k(i') - r(i')). \nonumber 
\end{align}
  Since $\|\mathbf{r}_k - \mathbf{r} \|_\infty$ decreases exponentially in $k$ and $\lim_{n\ra\infty} P_i( Z_n = k, \, R_{n+1}=i' ) = 0$ for every $k\geq 0$ and $i' \in \mathcal{R}$, it follows from the dominated convergence theorem that the last sum vanishes as $n\ra\infty$. 
  Since $\mathbf{r}_n \ra \mathbf{r}$ as $n\ra\infty$, we have shown that 
\begin{equation}\label{rZn}
 r(i)
 = \lim_{n\ra\infty} \left\{ \sum_{j=1}^n  \E_i\left[2p(R_j) - 1 \right]  +  e_i K^n \mathbf{r} \right\}
, \qquad \forall i  \in \mathcal{R}. 
\end{equation}
where the limit on the right necessarily converges. 
Similarly, 
a symmetric argument interchanging the roles of failures and successes yields 
\begin{equation}\label{trZn}
 \tilde{r}(i) 
= \lim_{n\ra\infty} \left\{   \sum_{j=1}^n  \E_i\left[1 - 2p(R_j) \right]  + e_i K^n \mathbf{\tilde{r}} \right\}
, \qquad \forall i \in \mathcal{R}. 
\end{equation}
Since clearly $Z_n + \tilde{Z}_n = n$ for all $n$, it follows from \eqref{rZn} and \eqref{trZn} that 
\[
 \mathbf{r} + \mathbf{\tilde{r}} = \lim_{n\ra\infty} K^n (\mathbf{r} + \mathbf{\tilde{r}} )
= \lim_{n\ra\infty} \frac{1}{n} \sum_{j=1}^n K^j  (\mathbf{r} + \mathbf{\tilde{r}} ) = (\one \mu)  (\mathbf{r} + \mathbf{\tilde{r}} ). 
\]
(Note that the second limit above is needed since we did not assume that the Markov chain $R_j$ is aperiodic on the closed irreducible subset $\mathcal{R}_0$.)
\end{proof}

\section{The non-critical case}\label{sec:noncrit}
In this section we consider the non-critical case $\bar{p} = \mu \cdot \mathbf{p} \neq 1/2$ and give the proof of Theorem \ref{th:noncrit}. 
The proof will be obtained through an analysis of the forward and backward BLP from Section \ref{sec:branching}. Theorems \ref{th:fbp} and \ref{th:bbp}, which were the basis for the proofs of Theorems \ref{th:rectran}--\ref{th:limdist}, cannot be directly applied in the non-critical case $\bar{p}\neq 1/2$. The main idea in this section will be to construct a coupling of the non-critical forward/backward BLP with a corresponding critical forward/backward BLP and then use this coupling and the connection of the BLP with the ERW to obtain the conclusions of Theorem \ref{th:noncrit}.  

\subsection{Coupling} Let $K$ be a transition probability matrix for a
Markov chain on ${\cal R}=\{1,2,\dots,N\}$ with some initial
distribution $\eta$. We will assume that this Markov chain satisfies
the assumptions stated at the beginning of Section 1.1. In particular,
it has a unique stationary distribution $\mu$ supported on a closed
irreducible set ${\cal R}_0\subseteq{\cal R}$. Suppose that we are given two functions
$p_0,p_1:{\cal R}\to(0,1)$ such that $p_1\ge p_0$.  Let
$\ve{p_i}:=(p_i(1),p_i(2),\dots,p_i(N))$, $\ve{i}=0,1$, and assume
that $\mu\cdot\ve{p_1}>\mu\cdot\ve{p_0}=1/2$.

Next, we expand the state space for the Markov chain to be
$\mathcal{\hat{R}} = \{1,2,\ldots,2N\}$, and for any $\e\in (0,1]$ we
let $\hat{K}_\e$ be a $2N\times 2N$ transition matrix and
$\mathbf{\hat{p}}$ be a column vector of length $2N$ given by
\begin{equation}\label{Keform}
 \hat{K}_\e = \left( \begin{array}{c|c} (1-\e)K & \e K \\ \hline O_N & K \end{array} \right)
 \quad\text{and}\quad
 \mathbf{\hat{p}} = \left( \begin{array}{c} \mathbf{p_1} \\ \mathbf{p_0} \end{array} \right),
\end{equation}
where $O_N$ denotes an $N\times N$ matrix of all zeros.  It is clear
that the unique stationary distribution for $\hat{K}_\e$ is given by
$\hat{\mu} = (\mathbf{0},\mu)$, where $\mathbf{0}$ is a row vector of
$N$ zeros and $\mu$ is the stationary distribution for the transition
matrix $K$.  Moreover,
$\hat{\mu} \cdot \mathbf{\hat{p}} = \mu \cdot \mathbf{p_0} = 1/2$ so
that $\hat{K}_\e$ and $\mathbf{\hat{p}}$ correspond to a critical ERW.

Let $\hat{U}^\e$ and $\hat{V}^\e$ be the associated forward and
backward BLP for the critical ERW corresponding to $\mathbf{\hat{p}}$, $\hat{K}_\e$,
 and with initial distribution $\hat{\eta}_\e = ((1-\e)\eta,\e \eta)$. Let
$U$ and $V$ be the forward and backward BLP for the non-critical case
corresponding to $\mathbf{p_1}$, $K$,  and the initial distribution
$\eta$.  We claim that for any $\e\in (0,1]$ these BLPs can be
coupled so that
\begin{itemize}
 \item $U_k \geq \hat{U}_k^\e$ for all $k\geq 0$; 
 \item $V_k \leq \hat{V}_k^\e$ for all $k\geq 0$. 
\end{itemize}
To this end, recall that $\P_\eta$ is the distribution for the i.i.d.\
family of Markov chains $\{R_j^x\}_{j\ge 0}$, $x\in\Z$, 
with transition matrix $K$ and initial condition $\eta$.  We can
enlarge the probability space so that the measure $\P_\eta$ also
includes an i.i.d.\ sequence $\{\gamma_\e^x\}_{x\in\Z}$ of
Geometric($\e$) random variables (that is
$\P_\eta( \gamma_\e^x = k ) = (1-\e)^{k} \e$ for $k \geq 0$) that is
also independent of the family of Markov chains $\{R_j^x\}_{j\geq 1}$,
$x\in\Z$.  We let $\w_x(j) = p_1(R_j^x)$ and also define a different
cookie environment
$\hat{\w}^\e = \{\hat{\w}^\e_x(j)\}_{x\in \Z, \, j\geq 1}$ by
\[
 \hat{\w}_x^\e(j) = \begin{cases} p_1(R_j^x) & j\leq \gamma_\e^x \\ p_0(R_j^x) & j > \gamma_\e^x. \end{cases}
\]
It is clear from this construction that the cookie environments $\w$
and $\hat{\w}$ are coupled so that $\hat{\w}_x^\e(j) \leq \w_x(j)$ for
all $x \in \Z$ and $j\geq 1$.  Given such a pair $(\w,\hat{\w}^\e)$ of
coupled cookie environments, let $\{\xi_x(j)\}_{x\in\Z, \, j\geq 1}$
and $\{\hat\xi^\e_x(j)\}_{x\in \Z,\, j\geq 1}$ be families of
independent Bernoulli random variables with
$\xi_x(j) \sim \text{Bernoulli}(\w_x(j))$ and
$\hat\xi_x^\e(j) \sim \text{Bernoulli}(\hat\w^\e_x(j))$ that are
coupled to have $\hat\xi_x^{\e}(j) \leq \xi_x(j)$ for all $x\in \Z$
and $j\geq 1$.  If $U$ and $V$ are the forward and backward BLP
constructed from the Bernoulli family $\{\xi_x(j)\}_{x,j}$ and
$\hat{U}^\e$ and $\hat{V}^\e$ are the forward and backward BLP
constructed from the Bernoulli family $\{\hat\xi_x^\e(j)\}_{x,j}$,
then the couplings $U_k \geq \hat{U}^\e_k$ and $V_k \leq \hat{V}_k^\e$
for all $k\ge 0$ follow immediately.  Moreover, it is easy to see that
$U$, $V$, $\hat{U}^\e$ and $\hat{V}^\e$ have the required
marginal distributions under this coupling.

As noted above, the forward and backward BLP $\hat{U}^\e$ and
$\hat{V}^\e$ are critical BLP to which Theorems \ref{th:fbp} and
\ref{th:bbp} apply. In the applications of these theorems, however,
one must replace the parameter $\d = \d(\mathbf{p},K,\eta)$ by
\[
 \hat{\d}_\e = \hat{\d}_\e(\mathbf{p_1},\mathbf{p_0},K,\eta) := \d(\mathbf{\hat{p}},\hat{K}_\e, \hat\eta_\e ).
\]
The following lemma shows, in particular, that the parameter
$\hat{\d}_\e$ can be made arbitrarily large by letting $\e \ra 0$.
\begin{lem}\label{depslim}
  Let $\epsilon\in(0,1]$, $\eta$ be an arbitrary distribution on
  ${\cal R}$, and $\hat{U}^\epsilon$ and $\hat{V}^\epsilon$ be the
  BLPs constructed above. Then
  \begin{equation}
    \label{deps}
    \hat\delta_\epsilon(\ve{p_1},\ve{p_0},K,\eta)=\delta(\ve{p_0},K,\eta)+\frac{4\eta\cdot\left( \sum_{j=1}^\infty(1-\epsilon)^jK^{j-1}(\ve{p_1}-\ve{p_0})\right) }{\nu(\ve{p_0},K)}.
  \end{equation}
In particular, $\hat\delta_\epsilon(\ve{p_1},\ve{p_0},K,\eta)$ is a
continuous strictly decreasing function of $\epsilon$ on $(0,1]$ which
equals $\delta(\ve{p_0},K,\eta)$ when $\e=1$ and
satisfies
\begin{equation}
  \label{eps0}
  \lim_{\e\downarrow 0}\ \hat\d_\e(\mathbf{p_1},\mathbf{p_0},K,\eta) = \infty.
\end{equation}
\end{lem}

\begin{proof}
It follows from the formula for $\d$ given in \eqref{eq:del} that 
\[
 \hat\d_\e(\ve{p_1},\ve{p_0},K,\eta) = \frac{2 ((1-\e)\eta,\e\eta) \cdot \mathbf{\hat{r}}_\e}{ \hat{\nu}_\e }, 
\quad \text{where } 
\mathbf{\hat{r}}_\e = \mathbf{r}( \mathbf{\hat{p}},\hat{K}_\e ) \text{ and } \hat\nu_\e = \nu( \mathbf{\hat{p}},\hat{K}_\e). 
\]
Therefore, if $\mathbf{\hat{r}}_\e =( \hat{r}_\e(1), \hat{r}_\e(2),\ldots,\hat{r}_\e(2N) )^t$ then \eqref{deps} will follow of we show that  for all $\epsilon\in(0,1]$
\begin{align}\label{nueconst}
 \hat\nu_\e& =  \nu( \mathbf{\hat{p}},\hat{K}_\e) = \nu( \ve{p_0}, K ),\ 
\text{and}\\
\label{relim}
\hat{\ve{r}}_\epsilon&=\hat{\ve{r}}(\mathbf{\hat{p}},\hat{K}_\e)= \bigg(\ve{r}(\ve{p_0},K)+2\sum_{j=1}^\infty(1-\epsilon)^jK^{j-1}(\ve{p_1}-\ve{p_0})  ,\ve{r}(\ve{p_0},K) \bigg)  
\end{align}
Note that \eqref{relim} states, in part, that the last $N$ coordinates of $\mathbf{\hat{r}}_\e$ are
equal to $\mathbf{r}(\ve{p_0},K)$ for any $\e\in (0,1]$. This follows easily
from the fact that the bottom right $N\times N$ block of the
transition matrix $\hat{K}_\e$ in \eqref{Keform} is equal to $K$.
Similarly, since the states $i=1,2,\ldots,N$ are transient for the
Markov chain with transition matrix $\hat{K}_\e$, it follows that the
stationary distribution $\tilde{\pi}$ corresponding to the pair
$(\mathbf{\hat{p}},\hat{K}_\e)$ is given by
$\tilde\pi(\mathbf{\hat{p}},\hat{K}_\e) =
(\mathbf{0},\tilde\pi(\ve{p_0},K))$.
Thus, the formula for the parameter $\nu$ in \eqref{parforms} implies
that
\[\hat{\nu}_\e = 2 + 2 \tilde{\pi}(\mathbf{\hat{p}},\hat{K}_\e) \cdot \mathbf{\hat{r}}_\e = 2 + 2 \tilde{\pi}(\ve{p_0},K)\cdot \mathbf{r}(\ve{p_0},K) = \nu(\ve{p_0},K).\]
This proves \eqref{nueconst}, and it remains to show \eqref{relim}  for the first $N$ coordinates of $\hat{\ve{r}}_\epsilon$.

Applying the formula for $r(i)$ from \eqref{rZn}, we have that 
\begin{align}
 \hat{r}_\e(i) &= \lim_{n\ra\infty} \bigg\{ \sum_{j=1}^n \E_i\left[ 2\left(p_1(R_j) \ind{j\leq \gamma_\e} + p_0(R_j)\ind{j>\gamma_\e} \right) - 1 \right] + e_i \hat{K}_\e^n \mathbf{\hat{r}}_\e \bigg\} \nonumber \\
&= \lim_{n\ra\infty} \bigg\{ \sum_{j=1}^n 2 (1-\e)^{j}  e_i K^{j-1} (\mathbf{p}_1 - \mathbf{p}_0)  + \sum_{j=1}^n\E_i \left[ 2 p_0(R_j) - 1 \right] +  e_i \hat{K}_\e^n \mathbf{\hat{r}}_\e \bigg\} \label{hatre1}
\end{align}
It follows from the form of the matrix $\hat{K}_\e$ and the fact that the last $N$ entries of the vector $\hat{r}_\e$ are equal to $\mathbf{r}(\ve{p_0},K)$ that
\[
 \lim_{n\ra\infty} \left| e_i \hat{K}_\e^n \mathbf{\hat{r}}_\e - e_i K^n \mathbf{r}(\ve{p_0},K) \right| = 0. 
\]
Together with \eqref{rZn} this implies that
 the limit of the last two terms in \eqref{hatre1} is the $i$-th entry of the vector $\mathbf{r}(\ve{p_0},K)$. 
Since the limit of the first term in \eqref{hatre1} is the corresponding infinite sum we obtain \eqref{relim}.

Finally, to prove \eqref{eps0}, it is enough to show that 
\begin{equation}\label{Kp0p1}
\sum_{j=0}^\infty \eta K^j (\mathbf{p}_0-\mathbf{p}_1) = \infty. 
\end{equation}
Recalling that $p_1\ge p_0$ and $\mu\cdot \ve{p_1}>\mu\cdot \ve{p_0}=1/2$ we conclude that there is a strict inequality $p_1(i) > p_0(i)$ for some state $i \in \mathcal{R}_0$ which is therefore a recurrent state for the Markov chain $R_j$. 
Since the $i$-th entry of $\sum_{j=0}^\infty \eta K^j$ equals the expected number of times the Markov chain $R_j$ visits the state $i$, \eqref{Kp0p1} follows. 
\end{proof}

\subsection{Proof of Theorem \ref{th:noncrit}}

Without loss of generality we may assume that $\bar{p} > 1/2$.  We
shall apply Lemma~\ref{depslim}. Choose $\ve{p_1}=\ve{p}$. Since
$\bar{p}>1/2$, there exists 
$\ve{p_0} = (p_0(1),p_0(2),\ldots,p_0(N))^t \in (0,1)^N$ such that
\begin{itemize}
 \item $p_0(i) \leq p_1(i)$ for all $i\in \mathcal{R}$. 
 \item $\mu \cdot \ve{p_0} = 1/2$. 
\end{itemize}
Let the distribution $\eta$ on $\mathcal{R}$ be fixed. By Lemma \ref{depslim} we may choose an $\e\in (0,1]$ so that $\hat{\d}_\e = \d_\e(\ve{p_1},\ve{p_0},K,\eta) > 4$, and we will keep this choice of $\e$ fixed for the remainder of the proof. 

\textbf{Transience:}
Since $\hat{\d}_\e > 1$, it follows from Theorem \ref{th:fbp} that $\s_0^{\hat{U}^\e} = \infty$ with positive probability for any initial condition $\hat{U}_0^{\e} = y \geq 1$ for the forward BLP $\hat{U}^\e$. Since the above coupling of the forward BLP is such that $U_k \geq \hat{U}^\e_k$ for all $k\geq 1$, this implies that $P_\eta^{U,y}( \s_0^U = \infty) > 0$ for any $y\geq 1$. From this, the proof that the ERW is transient to the right is the same as in the proof of Theorem \ref{th:rectran}. 

\textbf{Ballisticity and Gaussian limits:} 
Since $\hat{\d}_\e > 4$, it follows from Theorem \ref{th:bbp} that 
$\s_0^{\hat{V}^\e}$ and $\sum_{i=0}^{\s_0^{\hat{V}^\e}-1} \hat{V}_i^\e$ 
have finite second moments when the backward BLP $\hat{V}^\e$ is started with $\hat{V}^\e_0 = 0$. 
Since the above coupling is such that $V_k \leq \hat{V}^\e_k$ for all $k\geq 0$ this implies that
\begin{equation}\label{ncbbp}
 E_\eta^{V,0}[ (\s_0^V)^2 ] < \infty \quad\text{and}\quad E_\eta^{V,0}\bigg[ \bigg( \sum_{i=0}^{\s_0^V-1} V_i \bigg)^2 \bigg] < \infty. 
\end{equation}
As noted in the proof of Theorem \ref{th:LLN}, a formula for the limiting speed $\vp$ when the ERW is transient can be given by 
\[
 \frac{1}{\vp} = 1 + 2\, \frac{ E_\eta^{V,0}\left[  \sum_{i=0}^{\s_0^V-1} V_i\right] }{ E_\eta^{V,0}[ \s_0^V ] }. 
\]
Since \eqref{ncbbp} implies that the fraction on the right is finite, we can conclude that the limiting speed $\vp>0$. 
The proof of the Gaussian limiting distributions for the ERW is the same as for the proof of the limiting distributions of the critical ERWs in the case $\d>4$ since all that is needed are the finite second moments for the backward BLP given in \eqref{ncbbp}.

\section{Proofs of Theorems~\ref{th:fbp} and \ref{th:bbp}}\label{sec:3lem}

We shall discuss the proof of Theorem~\ref{th:bbp}, since
Theorem~\ref{th:fbp} can be derived in exactly the same way. The main
idea behind the proof of Theorem~\ref{th:bbp} is very simple. The
rescaled critical backward BLP  can be approximated (see Lemma~\ref{lem:DA}
below) by a constant multiple of a squared Bessel process of the
generalized dimension
\begin{equation}\label{Bsqdim}
\frac{4(1+\eta\cdot\tilde{\br})}{\tilde{\nu}}\overset{(\ref{eq:del}),\mathrm{L.\,}\ref{deleq}}
{=}\frac{4}{\nu}+2\tilde{\delta}\overset{\mathrm{L.\,}\ref{deleq}}
{=}2(1-\delta).
\end{equation}
The squared Bessel process of dimension $2$ is just the square of the
distance of the standard planar Brownian motion from the
origin, and this dimension is critical. For dimensions less than $2$
the squared Bessel process hits $0$ with probability $1$. For
dimensions $2$ or higher the origin is not attainable. Thus,
$\delta=0$ should be a critical value for $V$. This suggests that the
process dies out with probability 1 if $\delta>0$ and has a positive
probability of survival if $\delta<0$. Moreover, the tail decay
exponents for $\delta>0$ can also be read off from those of the
corresponding squared Bessel process. The boundary case $\delta=0$ is
delicate and has to be handled separately.

Nevertheless, a lot of work needs to be done to turn the above idea
into a proof. This is accomplished in \cite{kmLLCRW} and \cite{DK14}
for the model with bounded cookie stacks. The proofs of the respective
results in the cited above papers can be repeated {\it verbatim}
provided that we reprove several lemmas which depend on the specifics
of the backward BLP $V$. Therefore, below we only provide a rough
sketch of the proof and state several lemmas which depend on the
properties of our $V$. These lemmas cover all $\delta\in\R$ and are
used later in the section to discuss the three cases: $\delta>0$,
$\delta=0$, and $\delta<0$. Their proofs are given in the next
section.

\subsection{Sketch of the proof}
Let us start with the already mentioned approximation by the squared
Bessel process.\footnote{Following the number of each lemma in this
  subsection are the corresponding results in the literature which it
  replaces.}

\begin{lem}[Diffusion approximation, Lemma 3.1 of \cite{kmLLCRW}, Lemma 3.4 of \cite{DK14}]\label{lem:DA}  Fix an arbitrary $\epsilon>0$, 
  $y>\epsilon$, and a sequence $y_n\to y$ as $n\to\infty$. 
Define $Y^{\epsilon,n}(t)=\dfrac{V_{\floor{nt}\wedge \sigma_{\varepsilon n}}}{n}$, $t\ge 0$. 
Then, under $P_\eta^{V,ny_n}$ the process $Y^{\epsilon,n}$
  converges in the Skorokhod ($J_1$) topology  to
  $Y(\cdot\wedge \sigma^Y_\epsilon)$ where $Y$ is the solution of
\begin{equation}
\label{SqB2}
dY(t)=(1+\eta\cdot\tilde{\br})dt+\sqrt{\nu Y(t)}\, dW(t), \quad Y(0)=y.
\end{equation}
\end{lem}

\begin{rem*}
  For the ERW with bounded cookie stacks the convergence is known to
  hold up to $\sigma^Y_0$ (or in $D([0,\infty))$ if
  $\sigma^Y_0=\infty$) as long as the corresponding squared Bessel
  process has dimension other than $2$ (see \cite[Theorem 3.4]{kzEERW}
  for the forward branching process). But such result does not seem to
  significantly shorten the proof of Theorem~\ref{th:bbp}, and we
  shall not show it.
\end{rem*}

To prove (\ref{bbptails}) we need some tools to handle $V$ when it
starts with $y\ll \epsilon n$ or falls below $\epsilon n$. The idea
again comes from the properties of $Y$. It is easy to check that
$Y^\delta$ for $\delta\ne 0$ (and $\ln Y$ for $\delta=0$) is a local
martingale.  Then by a standard calculation we get that for all
$a>1$ and $j\in\N$
\[P(\tau^Y_{a^{j-1}}<\tau^Y_{a^{j+1}}|\,
Y(0)=a^j)=\frac{a^\delta}{1+a^\delta}.\]
The same power (logarithm for $\delta=0$) of the rescaled
process $V$ is close to a martingale, and we can prove a similar
result for all $\delta\in\R$ (for $\delta=0$ it is sufficient to
set $a=2$ below).
\begin{lem}[Exit distribution, Lemma 5.2 of \cite{kmLLCRW},
  Lemma~A.3 of \cite{DK14}]\label{lem:exit}
  Let $a>1$, $|x-a^j| \leq a^{2j/3}$, and $\gamma$ be the exit time
  from $(a^{j-1},a^{j+1})$. Then for all sufficiently large $j\in\N$
\[
\left| P_\eta^{V,x}( V_\gamma \leq a^{j-1} ) - \frac{a^\delta}{a^\delta + 1}
\right| \leq a^{-j/4}.
\]
\end{lem}
One of the estimates needed for the proof of Lemma \ref{lem:exit} is the following lemma which shows that the process $V$ exits the interval $(a^{j-1},a^j)$ not too far below $a^{j-1}$ or above $a^j$.

\begin{lem}[``Overshoot'', Lemma 5.1 of \cite{kmLLCRW}]\label{lem:OS}
  There are constants $\Cl[c]{ol}, \Cl[c]{ou}>0$ and $N\in\mathbb{N}$
  such that for all $x\geq N$, $y\geq 0$, and every initial
  distribution $\eta$ of the environment Markov chain
\[\max_{0\leq z<x} P^{V,z}_\eta(V_{\tau_x}>x+y\,|\,\tau_x<\sigma_0)\leq \Cr{ol} \left(e^{-\Cr{ou} y^2/x}+e^{-\Cr{ou} y}\right) \]
and
\[ \max_{x<z<4x} P^{V,z}_\eta(V_{\sigma_x\wedge \tau_{4x}}<x-y)\leq
\Cr{ol} e^{-\Cr{ou} y^2/x}.\]
\end{lem}


Lemma \ref{lem:OS} shows that the process $V$ typically exits the
interval $(a^{j-1},a^j)$ close enough to the boundary so that we can
repeatedly apply Lemma \ref{lem:exit} to couple $V$ with a
birth-and-death-like process to obtain the following estimate on exit
probabilities from large intervals $(a^\ell,a^u)$  and
  ultimately handle $V$ when it is below $\epsilon n$.

\begin{lem}[Lemma 5.3 of \cite{kmLLCRW} and Lemma
  A.1 of \cite{DK14}] \label{lem:main} For each $a\in(1,2]$ there is an $\ell_0\in\N$ and a small positive number $\lambda$
  such that if $\ell, m, u, x\in\N$ satisfy $\ell_0\le\ell<m<u$ and
  $|x-a^m|<a^{2m/3}$
  then \[\frac{h^-(m)-h^-(\ell)}{h^-(u)-h^-(\ell)}\le
    P^{V,x}_\eta(\sigma^V_{a^\ell}>\tau^V_{a^u})\le
    \frac{h^+(m)-h^+(\ell)}{h^+(u)-h^+(\ell)},\] where for $j\ge
      1$ 
\[h^\pm(j)=
      \begin{cases}
        \prod_{i=1}^j(a^\delta\mp a^{-\lambda i}),&\text{if }\delta>0;\\
      \prod_{i=1}^j(a^{-\delta}\mp a^{-\lambda i})^{-1},&\text{if }\delta<0;\\
      j\mp\frac1j,&\text{if }\delta=0.
      \end{cases}
      \] Moreover, for $\delta\ne 0$ there are $K_i:\N\to(0,\infty)$,
      $i=1,2$, such that $K_i(\ell)\to 1$ as $\ell\to \infty$ and for
      all $j>\ell$
      \begin{equation}
        \label{asymp}
        K_1(\ell)a^{(j-\ell)\delta}\le
      \frac{h^\pm(j)}{h^\pm(\ell)}\le K_2(\ell)a^{(j-\ell)\delta}.
      \end{equation}
\end{lem}
\begin{proof}
 The proof is essentially identical to the proof of \cite[Lemma 5.3]{kmLLCRW} and relies only on the properties proved already in Lemmas \ref{lem:exit}, \ref{lem:OS} and the fact that the BLP are naturally monotone with respect to their initial value $V_0$. 
The proof of \cite[Lemma 5.3]{kmLLCRW} was given for $\d>0$, but essentially the same proof holds for $\d\leq 0$ with 
only minor changes needed due to the form of the functions $h^{\pm}(j)$. 
\end{proof}

The proofs of Lemmas \ref{lem:DA}--\ref{lem:OS} will be given in Section \ref{sec:tails}. 
We close this section with a brief discussion of how to finish the proof of Theorem \ref{th:bbp} using these lemmas in the cases $\d>0$, $\d=0$ and $\d < 0$. 

\subsection{Case \texorpdfstring{$\delta>0$}{}} Substituting the
lemmas from the above subsection for the corresponding lemmas in
\cite{kmLLCRW} and repeating the proof given in this paper we obtain
(\ref{bbptails}) for $\delta>0$, which, in particular, implies that
$P^{V,y}_\eta(\sigma^V_0<\infty)=1$ for all $\delta>0$.

\subsection{Case \texorpdfstring{$\delta=0$}{}}
The proof of (\ref{bbptails}) for $\delta=0$ is identical to the proof
of \cite[Theorem~2.1]{DK14}. All the ingredients which depend on the
specifics of $V$ are contained in the lemmas above. Note again
that while the 2-dimensional squared Bessel process never hits zero,
(\ref{bbptails}) implies that $P^{V,y}_\eta(\sigma^V_0<\infty)=1$ for
$\delta=0$.

\subsection{Case \texorpdfstring{$\delta<0$}{}}\label{ssec:dl0}
All we need to show is that $P^{V,y}_\eta(\sigma^V_0=\infty)>0$ for
every $y\geq 0$.  Notice that by the strong Markov property and
monotonicity of the BLP with respect to the starting point for every
$m>\ell>\max\{\ell_0,\log_2(y+1)\}$
\begin{align*}
  P^{V,y}_\eta(\sigma^V_0=\infty)\ge
P^{V,y}_\eta(\sigma^V_0=\infty, \tau^V_{2^m}<\sigma^V_0)&\ge
P^{V,2^m}_\eta(\sigma^V_0=\infty)P^{V,y}_\eta (\tau^V_{2^m}<\sigma^V_0)\\&\ge
P^{V,2^m}_\eta(\sigma^V_{2^\ell}=\infty)P^{V,y}_\eta (\tau^V_{2^m}<\sigma^V_0).
\end{align*}
By the lower bound of Lemma~\ref{lem:main} and (\ref{asymp}) we can
choose and fix sufficiently large $m>\ell$ so
that 
\[
P^{V,2^m}_\eta(\sigma^V_{2^\ell}=\infty)=\lim_{u\to\infty}P^{V,2^m}_\eta(\sigma^V_{2^\ell}>\tau^V_{2^u})\ge
1-\frac{h^-(m)}{h^-(\ell)}>0.
\] 
Moreover, by the ellipticity of the environment,
$P^{V,y}_\eta (\tau^V_{2^m}<\sigma^V_0)>0$, and we conclude that
$P^{V,y}_\eta(\sigma^V_0=\infty)>0$ for every $y\geq 0$.

\section{Proofs of Lemmas~\ref{lem:DA}-\ref{lem:OS}}\label{sec:tails}

We shall prove a more general diffusion approximation result and then
derive from it Lemma~\ref{lem:DA}.
\begin{lem}[Abstract lemma]\label{al}
Let $b\in \mathbb{R}$, $ D>0$, and $Y(t)$,
$t\ge 0$, be a solution of
\begin{equation}
  \label{SDE}
  dY(t)=b \, dt+\sqrt{ D \, Y(t)^+}dW(t),\quad Y(0)=x>0,
\end{equation}
where $W(t)$, $t\ge 0$, is the standard Brownian motion\footnote{Let us remark
that $X(t):=4Y(t)/ D$ satisfies
$dX(t)=(4b/ D)dt+2\sqrt{X(t)^+}dB(t)$ and, thus, is a squared Bessel
process of the generalized dimension $4b/ D$.}.
Let integer-valued Markov chains $Z_n:=(Z_{n,k})_{k\ge 0}$ satisfy the
following conditions: 
\begin{enumerate}
\item \label{DAasm1} there is a sequence $N_n\in\mathbb{N}$,
$N_n\to\infty$, $N_n=o(n)$ as $n\to\infty$, $f:\mathbb{N}\to [0,\infty),\ f(x)\to 0$ as $x\to
    \infty$, and $g:\mathbb{N}\to [0,\infty),\ g(x)\searrow 0$ as $x\to
    \infty$ such that 
\begin{align*}
  &\mathrm{(E)}\quad |E(Z_{n,1}-Z_{n,0}\,|\,Z_{n,0}=m)-b|\le f(m\vee
  N_n);\\&\mathrm{(V)}\quad
  \Big|\frac{\mathrm{Var}(Z_{n,1}\,|\,Z_{n,0}=m)}{m\vee N_n}- D  \Big| \le g(m\vee N_n);
\end{align*}
 \item \label{DAasm2} for each $T,r>0$ 
\[
E\left(\max_{1\le k\le (Tn)\wedge \tau^{Z_n}_{rn}}(Z_{n,k}-Z_{n,k-1})^2\right)=o(n^2)\ \text{ as $n\to\infty$},
\] 
where $\tau_x^{Z_n} = \inf\{ k\geq 0: Z_{n,k} \geq x \}$. 
  \end{enumerate}
  Set $Z_{n,0}=\floor{nx_n}$, $x_n\to x$ as $n\to\infty$, and
  $Y_n(t)=Z_{n,\floor{nt}}/n$, $t\ge 0$. Then
 $Y_n\overset{J_1}{\Rightarrow} Y$ as $n\to\infty$,
where $\overset{J_1}{\Rightarrow}$ denotes convergence in distribution with respect to the Skorokhod ($J_1$) topology. 
\end{lem}
\begin{proof}
The proof is based on \cite[Theorem 4.1,\ p.\,354]{ekMP}. We need to
check the well-posedness of the martingale problem for
\[A=\left\{\left(f,Gf=\frac{D}{2}\,x^+\,\frac{\partial^2}{\partial
      x^2}+b\,\frac{\partial}{\partial x}\right):\,f\in
  C_c^\infty(\R)\right\}\]
on $C_\R[0,\infty)$ and conditions (4.1)-(4.7) of that theorem. The
well-posedness of the martingale problem follows from \cite[Corollary
3.4, p.\,295]{ekMP} and the fact that the existence and distributional
uniqueness hold for solutions of (\ref{SDE}) with arbitrary initial
distributions.\footnote{A more detailed discussion of (\ref{SDE}) can
  be found immediately following (3.1) in \cite{kzEERW}.}  Thus we
only need to check the conditions of the theorem in \cite{ekMP}. 

To this end, define processes $A_n(t)$ and $B_n(t)$ by 
\[ A_n(t):
  =\frac1{n^2}\sum_{k=1}^{\floor{nt}}\mathrm{Var}(Z_{n,k}\,|\,Z_{n,k-1});\quad 
  B_n(t):= \frac1n\sum_{k=1}^{\floor{nt}}E(Z_{n,k}-Z_{n,k-1}\,|\,Z_{n,k-1}).
\]
It is elementary to check that the processes
  $M_n(t):=Y_n(t)-B_n(t)$ and $M_n^2(t)-A_n(t)$, $t\ge 0$, are
  martingales for all $n\in\mathbb{N}$. 
Therefore, it is sufficient to check the following five conditions for any $T,r > 0$. 
\begin{align}\label{EKc3}
 &\lim_{n\ra\infty} E\left[ \sup_{t\leq T \wedge \tau_r^{Y_n}} \left| Y_n(t) - Y_n(t-) \right|^2 \right]  =0. 
\\
\label{EKc4}
  &\lim_{n\ra\infty} E\left[ \sup_{t\leq T \wedge \tau_r^{Y_n}} \left| B_n(t) - B_n(t-) \right|^2 \right]  =0. 
\\
\label{EKc5}
  &\lim_{n\ra\infty} E\left[ \sup_{t\leq T \wedge \tau_r^{Y_n}} \left| A_n(t) - A_n(t-) \right| \right]  =0. 
\\
\label{EKc6}
 &\lim_{n\ra \infty} \sup_{t \leq T \wedge \tau_r^{Y_n}} \left| B_n(t) - b t \right| = 0, \quad P\text{-a.s.}
\\
\label{EKc7}
  &\lim_{n\ra \infty} \sup_{t \leq T \wedge \tau_r^{Y_n}} \left| A_n(t) - D \int_0^t Y_n^+(s) \, ds  \right| = 0, \quad P\text{-a.s.}
\end{align}
Note, in the above conditions that  $\tau_r^{Y_n} = \inf\{t \geq 0: Y_n(t) \geq r \} = \tau_{rn}^{Z_n}/n$.  

Condition \eqref{EKc3} is simply a restatement of condition \ref{DAasm2} in the statement of Lemma \ref{lem:DA}. 
Conditions \eqref{EKc4} and \eqref{EKc5} follow from conditions \ref{DAasm1}(E) and \ref{DAasm1}(V), respectively. Indeed, 
\begin{align*}
   \lim_{n\ra\infty} E\left[ \sup_{t\leq T \wedge \tau_r^{Y_n}} \left| B_n(t) - B_n(t-) \right|^2 \right]
&= \lim_{n\ra\infty} \frac{1}{n^2} E\left[ \max_{1\leq k \leq (Tn) \wedge \tau_{rn}^{Z_n}} \left( E[ Z_{n,k} - Z_{n,k-1} \, | \, Z_{n,k-1}] \right)^2 \right] \\
&\overset{\text{(i)(E)}}{\leq} \lim_{n\ra\infty} \frac{(b+ \|f\|_\infty )^2}{n^2} = 0,
\end{align*}
and similarly
\begin{align*}
 \lim_{n\ra\infty} E\left[ \sup_{t\leq T \wedge \tau_r^{Y_n}} \left| A_n(t) - A_n(t-) \right| \right]  
&= \lim_{n\ra\infty} \frac{1}{n^2} E\left[ \max_{1\leq k \leq (Tn)\vee \tau_{rn}^{Z_n}} \Var(Z_{n,k} \, | \, Z_{n,k-1}) \right] \\
&\overset{\text{(i)(V)}}{\leq} \lim_{n\ra\infty} \frac{ (nr) \left( D + \|g\|_\infty \right) }{n^2} = 0, 
\end{align*}
where in the last inequality we used that $Z_{n,k-1} \leq rn$ for $k\leq \tau_{rn}^{Z_n}$ and that $N_n < rn$ for $n$ large enough. 
To check condition \eqref{EKc6}, note that 
\begin{align*}
 \sup_{t \leq T \wedge \tau_r^{Y_n}} \left| B_n(t) - bt \right| 
&\leq \frac{|b|}{n} + \sup_{1\leq k \leq (Tn) \wedge \tau_{rn}^{Z_n}} \frac{1}{n} \sum_{j=1}^k \left| E\left[ Z_{n,j} - Z_{n,j-1} \, | \, Z_{n,j-1} \right] - b \right| \\
&\leq \frac{|b|}{n} + \frac{1}{n} \sum_{j=1}^{(Tn) \wedge \tau_{rn}^{Z_n}} f(Z_{n,j-1} \vee N_n ) 
\leq \frac{|b|}{n} + T \left\{ \sup_{m \in [N_n,rn]} f(m) \right\}. 
\end{align*}
Since this final upper bound is deterministic and vanishes as $n\ra\infty$, we have shown that condition \eqref{EKc6} holds. 
Finally, to check condition \eqref{EKc7} note that 
\begin{align}
\left| A_n(t)\right. &-\left. D \int_0^t Y_n^+(s) \, ds  \right|
= \left| \frac{1}{n^2} \sum_{k=1}^{\fl{tn}} \Var(Z_{n,k} \, | \, Z_{n,k-1}) - \frac{D}{n^2} \sum_{k=1}^{\fl{tn}} Z_{n,k-1}^+ - \frac{D}{n}\left( t - \frac{\fl{tn}}{n} \right) Z_{n,\fl{tn}}^+ \right| \nonumber \\
&\leq \frac{1}{n^2} \sum_{\substack{1\leq k \leq \fl{tn}:\\ Z_{n,k-1} > N_n}} \left| \Var(Z_{n,k} \, | \, Z_{n,k-1}) - D Z_{n,k-1} \right| \nonumber \\
&\makebox[2.5cm]{\ } + \frac{1}{n^2} \sum_{\substack{1\le k\le \floor{tn}:\\Z_{n,k-1}\le N_n}} \left\{ \Var(Z_{n,k} \, | \, Z_{n,k-1}) + D Z_{n,k-1}^+ \right\}  
 +  \frac{D}{n}\left( t - \frac{\fl{tn}}{n} \right) Z_{n,\fl{tn}}^+ \nonumber \\
&\leq  \frac{1}{n^2} \sum_{\substack{1\leq k \leq \fl{tn}:\\ Z_{n,k-1} > N_n}} g(N_n) Z_{n,k-1} 
+ \frac{t N_n(2D + g(N_n))}{n} +  \frac{D}{n}\left( t - \frac{\fl{tn}}{n} \right) Z_{n,\fl{tn}}^+, \label{EKc7ub}
\end{align}
where the last inequality follows from condition \ref{DAasm1}(V) and the assumption that $g$ is non-increasing. 
Now if $t \leq \tau_r^{Y_n}$ (equivalently, $tn \leq \tau_{rn}^{Z_n}$) then $Z_{n,k-1} < rn$ for all $k\leq \fl{tn}$, and so the first sum in \eqref{EKc7ub} is at most $r\, t\, n^2 g(N_n)$. 
As for the last term in \eqref{EKc7ub}, if $t < \tau_r^{Y_n}$ then $Z_{n,\fl{tn}} < rn$ whereas if $t=\tau_r^{Y_n}$ then $tn$ is an integer and  the last term in \eqref{EKc7ub} is zero. Thus, we conclude that 
\[
 \sup_{t \leq T \wedge \tau_r^{Y_n}} \left| A_n(t) - D \int_0^t Y_n^+(s) \, ds  \right|
\leq r T g(N_n) + \frac{T N_n(2D + g(N_n))}{n} + \frac{Dr}{n}
\xrightarrow[n\ra\infty]{} 0. 
\]
This completes the proof of condition \eqref{EKc7} and thus also the proof of the lemma. 
\end{proof}
\begin{proof}[Proof of Lemma~\ref{lem:DA}]
  The diffusion approximation is obtained from Lemma~\ref{al} in two
  steps: (1) construct a modified process $\Bar{V}$
  for which it is easy to check the conditions of Lemma~\ref{al} and
  conclude the convergence to the diffusion for all times; (2) couple
  the original process  to the modified one so that they coincide up
  to the first time the processes enter $(-\infty,N_n]$. Since $N_n = o(n)$, this 
  gives the diffusion approximation up to the first entrance time to
  $(-\infty,n\epsilon]$ for every $\epsilon>0$ as claimed.

 Note  that the backward BLP $V_k$ can be written as
  $V_k = \sum_{m=1}^{V_{k-1}+1} \tilde{G}_m^k$, where
  $\tilde{G}_m^x = F_m^x - F_{m-1}^x$ is the number of failures
  between the $(m-1)$-st and $m$-th success in the sequence of
  Bernoulli trials $\{\xi_x(j)\}_{j\geq 1}$.  It follows that
\begin{equation}\label{Vkbpform}
 V_{k} = \sum_{m=1}^{V_{k-1}+1} \tilde{G}_m^{k} = V_{k-1} + 1 + \sum_{m=1}^{V_{k-1}+1} \left( \tilde{G}_m^{k} - 1 \right), \qquad k\geq 1. 
\end{equation}
We now construct the family of modified processes $\bar{V}_n = \{\bar{V}_{n,k} \}_{k\geq 0}$ as follows. 
Fix any sequence $N_n\in\mathbb{N}$ such that $2\le N_n\to\infty$
  and $N_n=o(n)$ as $n\to\infty$.  For all sufficiently large $n$ (we
  want $N_n\ll ny_n$) set
\begin{equation*}
  \Bar{V}_{n,0}=\floor{ny_n}, \quad\text{and}\quad \Bar{V}_{n,k}:=\Bar{V}_{n,k-1}+1+\sum_{j=1}^{(\Bar{V}_{n,k-1}+1)\vee N_n}(\tilde{G}_j^{k}-1),\quad k\ge 1. 
\end{equation*}
Note that the modified process $\bar{V}_n$ is naturally coupled with the BLP $V$ since we use the same sequence of Bernoulli trials in the construction of both. In particular, if we start with $\bar{V}_{n,0} = \fl{n y_n} > \e n > N_n$ then the two process are identical up until exiting $[\e n,\infty)$. 

It remains now to check the conditions of Lemma \ref{al} for the family of modified processes $\bar{V}_n$. 

\noindent
{\em Parameters and condition (i).} By Propositions~\ref{pr:Vmean} and
\ref{Vvarerror}, condition (i)(E) holds for $Z_n=\Bar{V}_n$ with
$b=1+\eta\cdot\tilde{\br}$ and $f(x)=\Cr{nl}e^{-\Cr{nu} x}$, and
condition (i)(V) holds with $D=\nu$ and $g(x)=\Cr{w}/x$. 

\noindent
{\em Condition (ii).} Fix $T,r>0$. We need to show
that
\[
\lim_{n\to\infty}\frac{1}{n^2}E_\eta\left[\max_{1\le k\le (Tn)\wedge
    \tau_{rn}^{\bar{V}_n}}\Big|1+\sum_{j=1}^{(\Bar{V}_{n,k-1}+1)\vee
    N_n}(\tilde{G}_j^{k}-1)\Big|^2\right]=0,
\] 
where $\tau_{rn}^{\bar{V}_n} = \inf\{k: \bar{V}_{n,k} \geq rn \}$. 
To see this, note that if $n$ is large enough so that $N_n < r n$ then 
\begin{align*}
  \frac{1}{n^2}E_\eta \left(\max_{1\le k\le (Tn)\wedge
      \tau_{rn}^{\bar{V}_n}}\right. &\left. \Big|\sum_{j=1}^{(\Bar{V}_{n,k-1}+1)\vee
      N_n}(\tilde{G}_j^{k}-1)\Big|^2\right)\le
  \frac{1}{n^2}E_\eta\left(\max_{1\le k\le (Tn)} \max_{N_n\le m\le
      rn+1}\Big|\sum_{j=1}^m(\tilde{G}_j^{k}-1)\Big|^2\right)\\ &\le
  \frac{1}{n^2} \sum_{y=0}^\infty P_\eta\left(\max_{1\le k\le (Tn)}
    \max_{N_n\le m\le
      rn+1}\Big|\sum_{j=1}^m(\tilde{G}_j^{k}-1)\Big|^2>y\right) 
\\&\le \frac{r^{3/2}}{\sqrt{n}}
+ r T \sum_{y\ge (rn)^{3/2}} P_\eta\left( \Big|\sum_{j=1}^m(\tilde{G}_j^{k}-1)\Big|>\sqrt{y}\right).     
  \end{align*}
  Finally we apply Lemma~\ref{lem:dev} to get that the expression in
  the last line does not exceed
\begin{align*}
  &\frac{r^{3/2}}{\sqrt{n}}+r T\sum_{y\ge (rn)^{3/2}} C
  \left(\exp\left\{-C' \left(\frac{y}{\sqrt{y}\vee
          (8rn)}\right)\right\} + \exp\left\{-C'\sqrt{y}\right\}\right)\\ \le
  &\frac{r^{3/2}}{\sqrt{n}}+r T\sum_{y\ge (rn)^{3/2}} C
  \left(\exp\left\{-C' \left(\frac{y}{\sqrt{y}\vee
          (8y^{2/3})}\right)\right\} + \exp\left\{-C'\sqrt{y}
    \right\}\right)\to 0\ \text{as } n\to\infty.
\end{align*}
By Lemma~\ref{al} we conclude 
that the processes $t\mapsto \bar{V}_{n,\fl{nt}}/n$ with initial conditions $\bar{V}_{n,0} = \fl{n y_n}$ and $y_n \ra y$ converge in distribution as $n\ra\infty$ to the process $Y$ defined by \eqref{SqB2}. 
As noted above, since $\bar{V}_n$ can be coupled with $V$ until time $\tau_{\e n}^{V}$, this gives us the desired result. 
\end{proof}
 Next we give the proof of the overshoot estimates in Lemma \ref{lem:OS} since they are needed for the proof of Lemma \ref{lem:exit}. 

\begin{proof}[Proof of Lemma \ref{lem:OS}]
The proof follows closely the arguments in the proof of Lemma 5.1 in \cite{kmLLCRW}. 
We begin by noting that 
\[
 \max_{0\leq z<x} P_\eta^{V,z}(V_{\tau_x^V} > x+y |\, \tau_x^V < \s_0^V )
\leq \max_{0\leq z < x} \frac{P_\eta^{V,z}( V_1 > x+y)}{P_\eta^{V,z}(V_1 \geq x)}. 
\]
Since the (averaged) distribution of the sequence $\{\tilde{G}_j^x\}_{j\geq 1}$ defined above doesn't depend on $x\in\Z$ we will let $\{\tilde{G}_j\}_{j\geq 1}$ denote a sequence with this same distribution. It follows from \eqref{Vkbpform} that 
\begin{align}
  \max_{0\leq z<x} P_\eta^{V,z}(V_{\tau_x^V} > x+y |\, \tau_x^V < \s_0^V )
&\leq \max_{0\leq z < x} \frac{ P_{\eta}\left( z+1 + \sum_{j=1}^{z+1} ( \tilde{G}_j - 1) > x+y \right) }{P_{\eta}\left( z+1 + \sum_{j=1}^{z+1} ( \tilde{G}_j - 1) \geq x \right)} \nonumber \\
&= \max_{0\leq m < x} \frac{ P_{\eta}\left(\sum_{j=1}^{x-m}  ( \tilde{G}_j - 1) > y+m \right) }{ P_{\eta}\left( \sum_{j=1}^{x-m}  ( \tilde{G}_j - 1) \geq m \right) }, \label{tGLDfrac}
\end{align}
where the last equality follows from the substitution $m=x-z-1$. 
We therefore need to give a lower bound on the denominator and an upper bound on the numerator of \eqref{tGLDfrac}. For the lower bound on the denominator we will use the following lemma.
\begin{lem}\label{lem:tGsum}
 \[
  \inf_{n\geq 1} \min_{i \in \mathcal{R}} P_i \left( \sum_{j=1}^n (\tilde{G}_j - 1) \geq 0 \right) > 0. 
 \]
\end{lem}
\begin{proof}
The event $\{ \sum_{j=1}^n (\tilde{G}_j - 1) \geq 0 \}$ occurs when there are at least $n$ failures by the time of the $n$-th success in the sequence of Bernoulli trials. This occurs if and only if there are at least $n$ failures in the first $2n-1$ trials, or equivalently less than $n$ successes in the first $2n-1$ trials. That is,  
 \[
  P_i \left( \sum_{j=1}^n (\tilde{G}_j - 1) \geq 0 \right) = P_i\left( \sum_{j=1}^{2n-1}  \xi(j) < n \right) 
= P_i\left( \sum_{j=1}^{2n-1}  \left( \xi(j) - \frac{1}{2} \right) \leq 0 \right).
 \]
The last probability converges to $1/2$ for each $i \in \mathcal{R}$ by Lemma~\ref{clts}.
\end{proof}

Next, we show how to use Lemma \ref{lem:tGsum} to obtain a lower bound on the denominator on the right side of \eqref{tGLDfrac}. 
For any $m\geq 1$ let $\rho_m = \min\{ n\geq 1: \, \sum_{j=1}^n (\tilde{G}_j - 1) \geq m \}$. 
Then, clearly
\begin{align}
 P_{\eta}\left( \sum_{j=1}^{x-m}  ( \tilde{G}_j - 1) \geq m \right) 
&= \sum_{n=1}^{x-m} P_{\eta}\left( \sum_{j=1}^{x-m}  ( \tilde{G}_j - 1) \geq m, \, \rho_m = n \right) \nonumber \\
&\geq \sum_{n=1}^{x-m} P_{\eta}\left( \sum_{j=n+1}^{x-m}  ( \tilde{G}_j - 1) \geq 0, \, \rho_m = n \right) 
\geq C P_{\eta} (\rho_m \leq x-m),  \label{OSdlb}
\end{align}
where the constant $C>0$ in the last line comes from Lemma
\ref{lem:tGsum}. 

For the numerator in \eqref{tGLDfrac}, 
\begin{align}
& P_{\eta}\left(\sum_{j=1}^{x-m}  ( \tilde{G}_j - 1) > y+m \right) \nonumber \\
&\leq \sum_{n=1}^{x-m}  P_{\eta}\left( \sum_{j=1}^n  ( \tilde{G}_j - 1) > m + \frac{y}{2}, \, \rho_m = n \right)  + \sum_{n=1}^{x-m} P_{\eta}\left( \sum_{j=n+1}^{x-m}  ( \tilde{G}_j - 1) > \frac{y}{2}, \, \rho_m = n \right)\nonumber \\
&\leq \sum_{n=1}^{x-m} \left\{ \max_{i \in \mathcal{R}} \max_{\ell < m}  P_i \left( \tilde{G}_1-1 > m-\ell + \frac{y}{2}  \, \biggl| \, \tilde{G}_1 - 1 \geq m-\ell \right) \right\} P_{\eta} ( \rho_m = n ) \nonumber \\
&\makebox[5cm]{\ } + \sum_{n=1}^{x-m} P_{\eta}(\rho_m=n) \left\{ \max_{i \in \mathcal{R}} \max_{k < x} P_i \left( \sum_{j=1}^{k} (\tilde{G}_j - 1) > \frac{y}{2} \right) \right\} \nonumber \\
&\leq \left\{ p_{\text{max}}^{\fl{y/2}} + \max_{i \in \mathcal{R}} \max_{k < x} P_i \left( \sum_{j=1}^{k} (\tilde{G}_j - 1) > \frac{y}{2} \right)  \right\} P_{\eta}(\rho_m \leq x-m), \label{OSnub}
\end{align}
where $p_{\text{max}} = \max_{i \in \mathcal{R}} p(i) < 1$.  Applying
\eqref{OSdlb} and \eqref{OSnub} to \eqref{tGLDfrac}, and using the
concentration
bounds in Lemma~\ref{lem:dev} finishes the proof
of the first part of Lemma \ref{lem:OS}.
The proof of the second inequality in Lemma \ref{lem:OS} is essentially the same and is therefore omitted. 
\end{proof}

\begin{proof}[Proof of Lemma~\ref{lem:exit}] 
As noted prior to the statement of Lemma \ref{lem:exit}, if $Y(t)$ is the process in \eqref{SqB2} that arises as the scaling limit of the backward BLP, then $Y^\d(t)$ (or $\log Y(t)$ when $\d=0$) is a martingale. The proof of Lemma \ref{lem:exit} is then accomplished by showing that $V_k^\d$ (or $\log V_k$ when $\d=0$) is nearly a martingale prior to exiting the interval $(a^{j-1},a_j)$. 
The proof follows essentially the same approach as in \cite[proof of Lemma 5.2(ii), pp.\ 598-599]{kmLLCRW}. 
To this end, let $s\in C^\infty([0,\infty))$ have
  compact support, and satisfy
\[
 s(x) = \begin{cases} x^\d & \d \neq 0 \\ \log x & \d = 0 \end{cases}
\qquad \text{for all } x \in \left( \frac{2}{3a}, \frac{3a}{2} \right).
\]
Define $\mathcal{H}_k:= \s( V_i: \, i\leq k )$. Then, a Taylor expansion for $s(x)$ yields that on the event $\{\gamma > k\}$
\begin{align*}
&E_\eta^{V,x}\left[ s\left( \frac{V_{k+1}}{a^j} \right) \biggl| \, \mathcal{H}_k \right] \\
&= s\left( \frac{V_{k}}{a^j} \right)
+ \frac{s'\left( \frac{V_{k}}{a^j} \right)}{a^j} E_\eta^{V,x}\left[ V_{k+1}-V_k \bigl| \, \mathcal{H}_k \right] 
+  \frac{s''\left( \frac{V_{k}}{a^j} \right)}{2 a^{2j}}E_\eta^{V,x}\left[(V_{k+1}-V_k)^2 \bigl| \, \mathcal{H}_k \right] + \e_{j,k}, 
\end{align*}
where the term $\e_{j,k}$ comes from the error in the Taylor expansion and can thus be bounded by 
\begin{equation}\label{ekjub}
 \e_{j,k} 
\leq \frac{\|s'''\|_\infty}{6 a^{3j}} E_\eta^{V,x}\left[ (V_{k+1}-V_k)^4 \biggl| \, \mathcal{H}_k \right]^{3/4}
\leq \frac{ \|s'''\|_\infty}{6 a^{3j}} (A V_k^2)^{3/4}
\leq C a^{-3j/2},
\end{equation}
where the second-to-last inequality follows from the fourth moment bound in Lemma \ref{lem:4th} and the last inequality follows from the fact that $V_k < a^{(j+1)}$ on the event $\{\gamma > k\}$. 
Next, it follows from Propositions \ref{pr:Vmean} and \ref{Vvarerror} that on the event $\{\gamma > k\}$
\begin{align}\label{Vkdiff}
 &\left| E_\eta^{V,x}\left[ V_{k+1}-V_k \bigl| \, \mathcal{H}_k \right] - (1+\eta \cdot \tilde{\mathbf{r}}) \right| \leq \Cr{nl} e^{-\Cr{nu} V_k} 
\leq \Cr{nl} e^{-\Cr{nu} a^{j-1}}
\\\intertext{and }
\label{Vksqdiff}
&\left| E_\eta^{V,x}\left[(V_{k+1}-V_k)^2 \bigl| \, \mathcal{H}_k \right] - \tilde{\nu} V_k \right|
\leq C''. 
\end{align}

Combining \eqref{ekjub}, \eqref{Vkdiff}, and \eqref{Vksqdiff} we see that
\begin{align}
  E_\eta^{V,x}\left[ s\left( \frac{V_{k+1}}{a^j} \right) \biggl| \, \mathcal{H}_k \right]
&= s\left( \frac{V_{k}}{a^j} \right)
+ \frac{1}{a^j} \left\{ (1+\eta \cdot \mathbf{\tilde{r}}) s'\left( \frac{V_{k}}{a^j} \right) 
+  \frac{\tilde{\nu}}{2}\frac{V_k}{a^{j}}   s''\left( \frac{V_{k}}{a^j} \right)  \right\} + R_{j,k}, \label{mgapprox}
\end{align}
where on the event $\{\gamma > k\}$ the error term $R_{j,k}$ is such that
\begin{equation}\label{Rjkub}
 |R_{j,k}| \leq C a^{-3j/2} + C a^{-j} e^{-C' a^j} + C a^{-2j} \leq C'' a^{-3j/2}. 
\end{equation}
Now, it follows from \eqref{Bsqdim} and the fact that $s(x) = x^\d$ in $[1/a,a]$ that 
\[
 (1+\eta \cdot \mathbf{\tilde{r}}) s'(x) + \frac{\tilde{\nu}}{2} x s''(x) 
= \frac{\tilde{\nu}}{2} \left\{ (1-\d) s'(x) +  x s''(x) \right\} 
= 0, \quad \forall x \in [1/a,a]. 
\]
Therefore, the quantity inside the braces in \eqref{mgapprox} vanishes
on the event $\{\gamma > k\}$ and so we can conclude that
$s( \frac{V_{n \wedge \gamma}}{a^j}) - \sum_{k=0}^{(n \wedge
  \gamma)-1} R_{j,n}$
is a martingale with respect to the filtration $\mathcal{H}_n$.  From
this point, the remainder of the proof is the same as that of Lemma 5.2 in \cite{kmLLCRW}.
We will  give a sketch and refer the reader to
\cite{kmLLCRW} for details.  First, the diffusion approximation in
Lemma \ref{lem:DA} can be used to show that
$E_\eta^{V,x}[ \gamma] \leq C a^j$ and thus it follows from the
optional stopping theorem and \eqref{Rjkub} that
\begin{equation}\label{Mgerror}
 \left|E_\eta^{V,x}\left[s\left( \frac{V_{\gamma}}{a^j} \right)\right] - s(x/a^j) \right|\leq E_\eta^{V,x}\left[ \left| \sum_{k=0}^{\gamma-1} R_{j,k} \right| \right] \leq C a^{-j/2}. 
\end{equation}
Next, the over/under-shoot estimates in Lemma \ref{lem:OS} are sufficient to show that 
\begin{equation}\label{Eserror}
 \left| E_\eta^{V,x}\left[ s\left( \frac{V_\gamma}{a^j} \right) \right] - P_\eta^{V,x}( V_\gamma \leq a^{j-1} ) a^{-\d} - P_\eta^{V,x}( V_\gamma \geq a^{j+1} ) a^{\d} \right| \leq C a^{-j/3}. 
\end{equation}
Finally, since $|x-a^j| \leq a^{2j/3}$, it follows that $|s(\frac{x}{a^j}) - 1| \leq C a^{-j/3}$. From this and the estimates in \eqref{Mgerror}--\eqref{Eserror} the conclusion of Lemma \ref{lem:exit} follows.  
\end{proof}

\appendix

\section{Appendix}

In this appendix we prove several technical estimates which are used in the analysis of the BLP in Section \ref{sec:tails}. We begin with the following concentration bound.


\begin{lem}\label{lem:dev}
 Let $\bar{p}=1/2$. There exist constants $C,C'>0$ and
  $y_0 < \infty$ such that for all $n\ge 1$, $y\ge y_0$, and for any
  initial distribution $\eta$ for the environment Markov chain we have
\[
 P_{\eta}\left(\left\vert \sum_{j=1}^n (\tilde{G}_j - 1)\right\vert > y \right) \leq C \left(\exp\left\{-C' \left(\frac{ y^2}{y\vee 8n}\right)\right\} + \exp\left\{-C'y  \right\}\right).
\] 
\end{lem}
Before giving the proof of Lemma \ref{lem:dev}, we note that a similar concentration bound is true for sums of i.i.d.\ random variables with finite exponential moments.

\begin{lem}[Theorem III.15 in \cite{pSOIRV}]\label{mdiid}
  Let $Y_1,Y_2,\ldots$ be a sequence of i.i.d.\ non-negative random
  variables with $E[Y_1] = \mu$ and $E[ e^{\l_0 Y_1}] < \infty$ for some
  $\l_0>0$. Then, there exists a constant $C>0$ such
  that
 \[
  P\left( \left| \sum_{k=1}^n Y_k - \mu n \right| \ge y \right) \leq \exp\left\{-C \frac{ y^2}{y\vee n} \right\} 
 \]
\end{lem}

\begin{proof}[Proof of Lemma \ref{lem:dev}]
  Let $i^*$ be any positive recurrent state, so that $\mu(i^*)>0$. Fix such an $i^*$ and set
  $J_0 = \inf\{ j\geq 1: R_j = i^* \}$, and for $k\geq 1$ let
$
 J_k = \inf\{ j > J_{k-1}: R_j = i^* \}. 
$
Also, for $k\geq 1$ let 
\[
 Y_k = \sum_{j=J_{k-1}}^{J_k-1} \xi_j. 
\]
Note that $\{Y_k\}_{k\geq 1}$ is an i.i.d.\ sequence with $E_{\eta}[Y_1] =
E_{\eta}[J_1-J_0]/2 = 1/(2 \mu(i^*))$.

Now, for any $y>0$ and any integers $n,N\geq 1$
\begin{align*}
  P_{\eta}\left( \sum_{j=1}^n (\tilde{G}_j - 1) > y \right)
  &= P_{\eta}\left( \sum_{j=1}^{2n+y}\xi_j < n \right) \\
  &\leq P_{\eta}\left( J_0 > \frac{y}{2} \right) + P_{\eta}\left( J_{N} -
    J_0 > 2n + \frac{y}{2} \right) + P_{\eta}\left( \sum_{k=1}^{N} Y_k <
    n \right).
\end{align*}
Now, let $N = \fl{ \mu(i^*)\left( \frac{y}{4} + 2 n \right) }$ so that 
\begin{equation*}
 2n + \frac{y}{2} 
 \geq NE [J_1-J_0] + \frac{y}{4}\quad\text{and}\quad
n 
 \leq N E[Y_1]  - \frac{y}{8} + \frac{1}{2\mu(i^*)}.
\end{equation*}
Therefore, for $y>0$ sufficiently large (so that $y/8-1/(2\mu(i^*)) >
y/9$; thus $y > 36/\mu(i^*)$ is sufficient) we have
\begin{align*}
  P_{\eta}\left( \sum_{j=1}^n (\tilde{G}_j - 1) > y \right)
  \leq P_{\eta}\left( J_0 > \frac{y}{2} \right) &+ P_{\eta}\left( J_{N} - J_0 > N E[J_1-J_0] + \frac{y}{4} \right) \\
  &+ P_{\eta}\left( \sum_{k=1}^{N} Y_k < N E[Y_1] - \frac{y}{9}
  \right).
\end{align*}
Since the cookie stack $\{\w_0(j)\}_{j\geq 1}$ is a finite state
Markov chain and $i^*$ is in the unique irreducible subset, then it
follows easily that there exists constants $C_1,C_2 >0$ such that
\[
 P_{\eta}\left( J_0 > \frac{y}{2} \right) \leq C_1 e^{-C_2 y}, 
\]
for all $y>0$ and for all distributions ${\eta}$ of the environment
Markov chain. From this it also follows that $J_{k+1}-J_k$ has
exponential tails, and since $Y_k \leq J_{k+1}-J_k$ then $Y_k$ also
has exponential tails. Therefore, we can apply Lemma \ref{mdiid} to
get that
\begin{align*}
   P_{\eta}\left( J_{N} - J_0 > N E[J_1-J_0] + \frac{y}{4} \right) &+ P_{\eta}\left( \sum_{k=1}^{N} Y_k < N E[Y_1] - \frac{y}{9} \right) \\
  &\leq e^{-C_3 \frac{y^2/16}{N \vee (y/4)} } + e^{-C_4 \frac{ y^2/81}{ N \vee (y/9) } } \leq 2 e^{-C_5 \frac{ y^2}{ (9 N) \vee y } } .
\end{align*}
Now, recalling the definition of $N$ (which depends on both $n$ and
$y$) we see that $N\le \mu(i^*)(8n\vee y)/2<(8n\vee y)/2$. Therefore,
we can conclude that
\[
P_{\eta}\left( \sum_{j=1}^n (\tilde{G}_j - 1) > y \right) \leq C_1 e^{-C_2 y} + 2
e^{-C_6 \frac{y^2}{ (8n) \vee y } },
\]
for all $n\ge 1$ all $y\ge \ceil{36/\mu(i^*)}=:y_0$. 

The proof of the upper bound for the left tails is
similar. Indeed, \[P_{\eta}\left( \sum_{j=1}^n (\tilde{G}_j - 1) <-y
\right)=P_{\eta}\left( \sum_{j=1}^n\tilde{G}_j <n-y \right).  \] If $y\ge n$
then the probability is 0 and there is nothing to prove. Assume that
$n>y$. Then the last probability is equal to \[P_{\eta}\left(
  \sum_{j=1}^{2n-y-1}(1-\xi_i) <n-y \right)=P_{\eta}\left(
  \sum_{j=1}^{2(n-y)+y-1}(1-\xi_i) <(n-y) \right).\] This probability
can be handled in exactly the same way as the upper bound for the right
tails by setting $n'=n-y$. Then in the exponent we shall have (for $y<n$) $-y^2/(8(n-y)\vee y)\le -y^2/(8n)=-y^2/(8n \vee y)$. 
\end{proof}

The concentration bounds in Lemma \ref{lem:dev} can be used to prove the following moment bound for sums of the $\tilde{G}_i$.

\begin{lem}\label{lem:4th}
Let $\bar{p}=1/2$.
Then there is a constant $A=A(y_0,C,C')>0$ (see Lemma~\ref{lem:dev})
such that for any initial distribution $\eta$ on the environment
Markov chain and all $n\ge
1$ \[E_{\eta}\left[\left(\sum_{i=1}^n(\tilde{G}_i-1)\right)^4 \right]\le An^2.\]
\end{lem}
\begin{proof}
  We shall use the fact that for a non-negative integer-valued random
  variable $X$ 
  \[E(X^4)\le
  1+c\sum_{y=1}^\infty y^3P(X>y).\] Set
  $X=\Big|\sum_{i=1}^n(\tilde{G}_i-1)\Big|$. Then by Lemma~\ref{lem:dev} for all
  $n>y_0/8$
  \begin{align*}
    E(X^4)&\le 1+c\sum_{y=1}^{y_0-1}y^3P(X>y)+C''\sum_{y=y_0}^\infty
    y^3 (e^{-C'y^2/(8n\vee y)}+e^{-C'y})\\ &\le
    K_1(y_0,C,C')+C''\sum_{y=y_0}^{8n} y^3
    e^{-C'y^2/(8n)}+2C''\sum_{y=y_0}^\infty y^3 e^{-C'y}\\ &\le
    K_2(y_0,C,C')+C''\sum_{y=y_0}^{8n} y^3 e^{-C'y^2/(8n)}.
  \end{align*}

Since the function $y^3 e^{-C'y^2/8}$ is directly
Riemann integrable, we can bound the sum in the last line by $n^2$
times a Riemann sum approximation. That is,
\begin{align*}
  \sum_{y=y_0}^{8n} y^3 e^{-C'y^2/(8n)}
  &\leq n^2 \sum_{y=0}^{\infty} \left( \frac{y}{\sqrt{n}} \right)^3 e^{-\frac{C'}{8} \left( \frac{y}{\sqrt{n}} \right)^2 } \frac{1}{\sqrt{n}}\\
  &\sim n^2 \int_0^\infty y^3 e^{-C'y^2/8} \, dy, \qquad \text{as }
  n\ra\infty.
\end{align*}
\end{proof}

\begin{lem}\label{clts}
 Assume that $\bar{p}=1/2$. 
There exists a constant $v>0$ such that for any initial distribution $\eta$, under the measure $P_\eta$ we have that $\frac{1}{\sqrt{n}} \sum_{i=1}^n (\xi_i-1/2) \Ra v Z$, where $Z$ is a standard normal random variable. 

\end{lem}
\begin{proof}
  As in the proof of Lemma \ref{lem:dev}, fix a recurrent state
  $i^* \in \mathcal{R}$ and let $J_0,J_1,J_2,\ldots$ be the successive
  return times of the Markov chain to $i^*$. 
Set \[Z_k:=\sum_{j=J_{k-1}}^{J_k-1}(\xi_j-1/2),\quad k\ge 1.\]
$(Z_k)_{k\ge 1}$ is an i.i.d.\ sequence of centered square integrable
random variables. Let $L_n=0$ if $J_0>n$ and
$L_n=\max\{k\in\mathbb{N}: J_{k-1}\le n\}$ otherwise. Then
  \begin{align*}
    \frac{1}{\sqrt{n}}\bigg|\sum_{i=1}^n
      (\xi_i-1/2)-\sum_{k=1}^{L_n}Z_k\bigg| &=\frac{1}{\sqrt{n}}
    \bigg|\sum_{i=1}^{J_0-1}(\xi_i-1/2)-\sum_{i=n+1}^{L_{n+1}-1}(\xi_i-1/2)\bigg|\\
    &\le \frac{1}{2\sqrt{n}}\left(J_0+ \max_{1\le k\le
        n+1}(J_k-J_{k-1})\right)
  \end{align*}
Since $J_0$ is almost surely finite and $\{J_k-J_{k-1}\}_{k\geq 1}$ is i.i.d.\ with finite second moment, it follows that the last expression above converges to 0 in probability. 
Therefore, it remains only to prove that 
\[
 \frac{1}{\sqrt{n}}\sum_{k=1}^{L_n}Z_k\Rightarrow v Z
\]
for some $v > 0$.  By the law of large numbers,
$L_n/n \ra 1/E_{i^*}[J_0] = \mu(i^*)$ a.s., and the desired conclusion
follows from \cite[Theorem I.3.1(ii), p.17]{gSRW} with
$v^2 = \mu(i^*)E_{\eta}[ Z_1^2]$.  
\end{proof} 

\section{Details of the limiting distribution in the case \texorpdfstring{$\d=2$}{}}\label{sec:delta2}

In this appendix we will give the details of the proof of the limiting
distributions in Theorem \ref{th:limdist} in the case $\d=2$.  This is
the most subtle case in Theorem \ref{th:limdist} and has not been
fully written down for ERWs even for the model with boundedly many
cookies per site.  Even though the proof is very similar to that for
one-dimensional random walk in random environment for the analogous
case (\cite[pp.\ 166-168]{kksStable}), some of the details are
different due to the fact that the regeneration times $r_i$ have
infinite second moment.

First, recall the definition (\ref{stablecf}) of the $\a$-stable
  distribution $L_{\a,b}$ and the fact that if $Z$ has distribution
  $L_{\a,b}$ with $\a\neq 1$, then $cZ$ has distribution
  $L_{\a,b c^\a}$ for any $c>0$.  However, if $\a=1$ then a
  re-centering is needed to get a distribution of the same form.  This
  difference is indicative of the fact that for $\a\neq 1$ there is a
  natural choice of the centering for totally asymmetric $\a$-stable
  distributions: the distributions $L_{\a,b}$ have mean zero when
  $\a>1$ and have support equal to $[0,\infty)$ when
  $\a<1$. In contrast, when $\a=1$ there is no canonical choice of the
  centering for the totally asymmetric $1$-stable distributions. To
  this end, for $b>0$ and $\xi \in \R$ let $L_{1,b,\xi}$ be the
  probability distribution with characteristic exponent given by
\[
 \log \int_\R e^{iux} L_{1,b,\xi}(dx) = i u \xi - b |u|\left( 1 + \frac{2i}{\pi} \log|u| \sgn(u) \right). 
\]
That is, the distribution $L_{1,b,\xi}$ differs from $L_{1,b}$ by a simple spatial translation. 

Recall that $r_k$ is the time of the $k$-th return of
the backward BLP $V_i$ to zero and that
$W_k = \sum_{i=r_{k-1}}^{r_k-1} V_i$.  When $\d=2$, it follows from
Theorem \ref{th:bbp} that
\begin{equation}\label{r1S1tails}
 P_\eta^{V,0}( r_1 > t ) \sim c_3 t^{-2} \quad \text{and}\quad P_\eta^{V,0}( W_1 > t ) \sim c_4 t^{-1}, \quad \text{as } t\ra\infty, 
\end{equation}
for some constants $c_3 = c_3(0,\eta) > 0$ and $c_4 = c_4(0,\eta)> 0$.
Let $m(t) = E_\eta^{V,0}[ W_1 \ind{W_1 \leq t}]$ be the truncated
first moment of $W_1$ (note that the above tail asymptotics of $W_1$
imply that $m(t) \sim c_4 \log t$).  Then, it follows from 
\cite[Theorem 3.7.2]{dPTE} or
\cite[Theorem 17.5.3]{Feller2} that there exist constants $b'>0$ and
$\xi'\in\R$ such that
\begin{equation}\label{WkLD}
 \lim_{n\ra\infty} P_\eta^{V,0}\left( \frac{\sum_{k=1}^n W_k - n \, m(n)}{n} \leq x \right) = L_{1,b',\xi'}(x), \quad x\in \R,
\end{equation}
and that there exists a constant $A>0$ such that 
\begin{equation}\label{rkLD}
 \lim_{n\ra\infty} P_\eta^{V,0}\left( \frac{r_n - n\bar{r}}{A \sqrt{n \log n} } \leq x \right) = \Phi(x), \quad x\in \R. 
\end{equation}

Since \eqref{rkLD} implies that $r_{\fl{n/\bar{r}}}$ is typically close to $n$, we wish to approximate the distribution of $n^{-1} \sum_{i=1}^n V_i$ by that of $ n^{-1} \sum_{k=1}^{n/\bar{r}} W_k$. 
Indeed, we claim that the difference converges to zero in $P_\eta^{V,0}$-probability. 
To see this, note that for any $\e>0$, 
\begin{align*}
 P_\eta^{V,0}\left( \left| \sum_{i=1}^n V_i - \sum_{k=1}^{n/\bar{r}} W_k \right| \geq \e n \right)
&\leq P_\eta^{V,0} \left( \left| r_{\fl{n/\bar{r}}} - n  \right| > n^{3/4} \right) + 
P_\eta^{V,0} \left( \sum_{|k-\frac{n}{\bar{r}}| \leq n^{3/4}+1 } W_k \geq \e n \right) \\
&= P_\eta^{V,0} \left( \left| r_{\fl{n/\bar{r}}} - n  \right| > n^{3/4} \right) + 
P_\eta^{V,0} \left( \sum_{k=1}^{2 \fl{n^{3/4}}+3} W_k \geq \e n \right).
\end{align*}
It follows from \eqref{WkLD} and \eqref{rkLD} that both terms in the last line above vanish as $n\ra\infty$ for any $\e>0$. 
Therefore, we can conclude from this, the limiting distribution in \eqref{WkLD}, and the fact that $T_n$ has the same limiting distribution as $n + 2\sum_{i=1}^n V_i$ that  
\begin{multline*}
 \lim_{n\ra\infty} P_\eta^{V,0}\left( \frac{T_n - 2 \frac{n}{\bar{r}} m(\frac{n}{\bar{r}}) }{n} \leq x \right) 
= \lim_{n\ra\infty} P_\eta^{V,0} \left( \frac{n + 2\sum_{i=1}^n V_i - 2\frac{n}{\bar{r}} m(\frac{n}{\bar{r}}) }{ n } \leq x \right) \\
= \lim_{n\ra\infty} P_\eta^{V,0} \left( \frac{\sum_{k = 1}^{\fl{n/\bar{r}}} W_k - \frac{n}{\bar{r}} m(\frac{n}{\bar{r}}) }{ n/\bar{r} } \leq \frac{(x-1)\bar{r}}{2} \right) 
= L_{1,b',\xi'}\left( \tfrac{(x-1)\bar{r}}{2} \right)
= L_{1,b,\xi}(x), 
\end{multline*}
where in the last equality we have $b = \frac{2b'}{\bar{r}}$ and $\xi = 1 + \frac{2\xi'}{\bar{r}} - \frac{4 b'}{\pi \bar{r}} \log(\frac{2}{\bar{r}})$. 
This completes the proof of the limiting distribution for $T_n$ when $\d=2$ with $b>0$ as above,  $D(n) = \xi + \frac{2}{\bar{r}}m(\frac{n}{\bar{r}})$, and $a = \bar{r}/(2c_4)$. 

Before proving the limiting distribution for $X_n$, we first need to
remark on the specific choice of the centering term $D(n)$ in the
limiting distribution for $T_n$. One cannot use an arbitrary centering
term growing asymptotically like $a^{-1} \log n$, but the above choice
of the centering term by the function
$D(t) = \xi + \frac{2}{\bar{r}} E[ W_1 \ind{W_1 \leq t/\bar{r}} ]$
grows regularly enough due to the tail asymptotics of $W_1$ so that
\begin{equation}\label{Dndiff}
 \lim_{n\ra\infty} D(m_n) - D(n) = 0, \quad \text{if } 
m_n \sim n \text{ as } n\ra\infty. 
\end{equation}
For $t>0$ let $\Gamma(t) = \inf\{ s>0: \, s D(s) \geq t \}$. Then, it follows that 
\begin{equation}\label{Gtsym}
 \Gamma(t) \sim \frac{a t}{\log(t)} 
\qquad \text{and}\qquad \Gamma(t) D(\Gamma(t)) = t + o(\Gamma(t)), \quad \text{as } t\ra\infty. 
\end{equation}
The first asymptotic expression in \eqref{Gtsym} follows easily from the definition of  $\Gamma(t)$ and the fact that $D(t) \sim a^{-1} \log t$. 
For the second asymptotic expression in \eqref{Gtsym}, note that for $s$ sufficiently large $s D(s)$ is strictly increasing and right continuous in $s$. 
Moreover, if $s D(s)$ is discontinuous at $s_0$ then the size of the jump discontinuity is 
$2 (s_0/\bar{r})^2 P(W_1 = s_0/\bar{r} )$.  Therefore, 
\[| \Gamma(t) D(\Gamma(t)) - t | \leq 2
\left(\frac{\Gamma(t)}{\bar{r}}\right)^2 P\left( W_1 =
  \frac{\Gamma(t)}{\bar{r}} \right),\]
since if $s D(s)$ is continuous at $s=\Gamma(t)$ then the
above difference is zero while if $s D(s)$ is
discontinuous at $s=\Gamma(t)$ then the difference is at most the size
of the jump discontinuity at that point.  Since the tail asymptotics
of $W_1$ imply that $x P(W_1 = x) = o(1)$, the second asymptotic
expression in \eqref{Gtsym} follows.

Now, for any $x \in \R$ and $n\geq 1$ let $m_{n,x}:= \lceil \Gamma(n) + \frac{xn}{(\log n)^2} \rceil \vee 0$. 
Since $m_{n,x}$ grows asymptotically like $\Gamma(n)$ as $n\ra\infty$ it follows that 
\begin{multline}
\lim_{n\ra\infty} \frac{n -  m_{n,x} D(m_{n,x})}{m_{n,x}} 
= \lim_{n\ra\infty} \frac{n -  m_{n,x} D(\Gamma(n))}{m_{n,x}} \\
= \lim_{n\ra\infty}  \frac{n -  \left( \Gamma(n) + \frac{xn}{(\log n)^2} \right) D(\Gamma(n))}{\Gamma(n) + \frac{xn}{(\log n)^2} } 
= \lim_{n\ra\infty}  \frac{ -  \frac{xn}{(\log n)^2} D(\Gamma(n))}{\Gamma(n) + \frac{xn}{(\log n)^2} } 
=  \frac{-x}{a^2}, \label{mnxlim}
\end{multline}
where the first equality follows from \eqref{Dndiff} and the last two equalities follow from \eqref{Gtsym}.
Similarly, letting $M_{n,x} = m_{n,x} + \ceil{\sqrt{n}}$ it follows that
\begin{equation}\label{Mnxlim}
\lim_{n\ra\infty} \frac{n -  M_{n,x} D(M_{n,x})}{M_{n,x}} 
= \frac{-x}{a^2}.
\end{equation} 

Finally, 
it follows from \eqref{TXT} and \eqref{backtrack} that 
\begin{multline*}
P_\eta\left( \frac{T_{m_{n,x}} - m_{n,x} D(m_{n,x})}{m_{n,x}} > \frac{n -  m_{n,x} D(m_{n,x})}{m_{n,x}} \right) \\
 = P_\eta( T_{m_{n,x}} > n ) \leq P_\eta\left( \frac{X_n - \Gamma(n)}{n/(\log n)^2} < x \right)\leq P_\eta( T_{M_{n,x}} > n ) + \bigo(n^{-1/2}) \\
 = P_\eta\left( \frac{T_{M_{n,x}} - M_{n,x} D(M_{n,x})}{M_{n,x}} > \frac{n -  M_{n,x} D(M_{n,x})}{M_{n,x}} \right) + \bigo(n^{-1/2}).
\end{multline*}
From the limiting distribution for $T_n$, together with \eqref{mnxlim} and \eqref{Mnxlim}, we conclude that the first and the last probabilities in the display above converge to $1-L_{1,b}(-x/a^2)$. 

\bibliographystyle{alpha}
\bibliography{CookieRW}



\end{document}